\documentclass[twoside,11pt,a4paper]{article}
\usepackage{fullpage}
\usepackage{graphicx}
\usepackage{amsmath,amsthm}
\usepackage{amssymb,amsfonts,amsbsy}
\usepackage{dsfont,eufrak,mathrsfs}
\usepackage{calligra}
\usepackage{color}
\usepackage[T1]{fontenc}

\DeclareMathAlphabet{\mathpzc}{OT1}{pzc}{m}{it}
\DeclareFontShape{T1}{calligra}{m}{n}{<->s*[1.2]callig15}{}
\DeclareMathAlphabet{\mathcalligra}{T1}{calligra}{m}{n}

\newtheorem{Def}{Definition}[section]
\newtheorem*{Teono}{Theorem}

\newtheorem*{Corno}{Corollary}

\newtheorem{Lema}[Def]{Lemma}
\newtheorem{Prop}[Def]{Proposition}
\newtheorem{Cor}[Def]{Corollary}

\newtheorem{Rem}[Def]{Remark}

\newtheorem{Theo}{Theorem}

\newtheorem{Coro}[Theo]{Corollary}

\newcommand*{\fun}[3]{#1: #2 \rightarrow #3}
\newcommand*{\ccomp}[1]{\mathcal{C}(#1)}

\newcommand*{\acomp}[1]{\mathcal{A}(#1)}

\newcommand*{\EMod}[1]{\mathrm{Mod}^{*}(#1)}
\newcommand*{\stab}[1]{\mathrm{stab}_{pt}(#1)}
\newcommand*{\nat}{\mathbb{N}}
\newcommand*{\X}{\mathfrak{X}}
\newcommand*{\Cf}{\mathscr{C}}
\newcommand*{\Cfn}{\Cf_{0}}
\newcommand*{\Cff}{\Cf_{f}}
\newcommand*{\Bf}{\mathscr{B}}

\newcommand*{\Sf}{\mathscr{S}}
\newcommand*{\Mf}{\mathscr{A}}

\newcommand*{\Df}{\mathscr{D}}

\newcommand*{\N}[3]{N_{#1}^{#2, \hspace{0.05cm} #3}}
\newcommand*{\lk}[1]{\mathcalligra{lk} \hspace{0.12cm}(#1)}

\newcommand*{\eps}[2]{\epsilon^{#1, \hspace{0.05cm} #2}}

\newcommand*{\ColonEqq}{\mathrel{\mathop:}=}

\newcommand*{\wrt}{with respect to }

\newcommand*{\Aut}[1]{\mathrm{Aut}(#1)}

\author{Jes\'{u}s Hern\'{a}ndez Hern\'{a}ndez}
\title{Edge-preserving maps of curve graphs}
\date{}

\begin{document}
\maketitle
\begin{abstract}
 Suppose $S_{1}$ and $S_{2}$ are orientable surfaces of finite topological type such that $S_{1}$ has genus at least $3$ and the complexity of $S_{1}$ is an upper bound of the complexity of $S_{2}$. Let $\fun{\varphi}{\ccomp{S_{1}}}{\ccomp{S_{2}}}$ be an edge-preserving map; then $S_{1}$ is homeomorphic to $S_{2}$, and in fact $\varphi$ is induced by a homeomorphism. To prove this, we use several simplicial properties of rigid expansions, which we prove here.
\end{abstract}
\section*{Introduction}
\indent In this work we suppose $S_{g,n}$ is an orientable surface of finite topological type with genus $g \geq 3$ and $n \geq 0$ punctures. The extended mapping class group of $S_{g,n}$, denoted by $\EMod{S_{g,n}}$ is the group of isotopy classes of self-homeomorphisms of $S_{g,n}$.\\
\indent In 1979 (see \cite{Harvey}) Harvey defined the curve complex of a surface as the simplicial complex whose vertices are isotopy classes of essential curves on the surface, and whose simplices are defined by disjointness (see Section \ref{prelim} for the details). We call its $1$-skeleton the \textit{curve graph of} $S_{g,n}$, which we denote by $\ccomp{S_{g,n}}$.\\
\indent There is a natural action of $\EMod{S_{g,n}}$ on the curve graph, by automorphisms. In \cite{Ivanov} Ivanov proved that for genus at least $2$ every automorphism of the curve graph is induced by a homeomorphism of $S_{g,n}$. These results were extended for most other surfaces of finite topological type by Korkmaz and Luo in \cite{Korkmaz} and \cite{Luo}, respectively.\\
\indent Later on, Irmak (see \cite{Irmak1}, \cite{Irmak2}, \cite{Irmak3}), Behrstock and Margalit (see \cite{BehrMar}), and Shackleton (see \cite{Shack}) generalised these results for larger classes of simplicial maps. In particular, Shackleton's result implies that any locally injective self-map of the curve graph is induced by a homeomorphism for surfaces of high-enough complexity.\\
\indent Thereafter, Aramayona and Leininger introduced in \cite{Ara1} the concept of a rigid set of the curve graph (described below in a more general setting) and construct a finite rigid set for any orientable surface of finite topological type. Afterwards, they introduce in \cite{Ara2} a way of creating supersets from given sets, such that the supersets are capable of inheriting the property of being rigid (which is not trivial). This method is called the \textit{rigid expansion} of a set in \cite{Thesis} and \cite{JHH1} due to this property.\\
\indent In this article we use techniques similar to those shown in \cite{Shack} along with simplicial properties of the rigid expansions to obtain the following result, recalling that the complexity of a surface is denoted by $\kappa(S_{g,n}) = 3g-3+n$.
\begin{Theo}\label{TheoB}
 Let $S_{1} = S_{g_{1},n_{1}}$ and $S_{2} = S_{g_{2},n_{2}}$ be two orientable surfaces of finite topological type such that $g_{1} \geq 3$, and $\kappa(S_{2}) \leq \kappa(S_{1})$; let also $\fun{\varphi}{\ccomp{S_{1}}}{\ccomp{S_{2}}}$ be an edge-preserving map. Then, $S_{1}$ is homeomorphic to $S_{2}$ and $\varphi$ is induced by a homeomorphism $S_{1} \rightarrow S_{2}$.
\end{Theo}
\indent Note that in the context of graph theory, an edge-preserving map is a graph homomorphism. Thus, with Theorem \ref{TheoB} we generalise (for surfaces of genus at least $3$) Shackleton's result which requires the maps to be locally injective (see \cite{Shack}).\\
\indent To prove Theorem \ref{TheoB}, in Section \ref{chap3} we first take the simplicial interpretation of a rigid expansion and generalise it to the setting of abstract simplicial graphs:\\
\indent Let $\Gamma$ be a connected simplicial graph, $v$ be a vertex of $\Gamma$ and $B$ be a set of vertices of $\Gamma$. We say $v$ is uniquely determined by $B$ if $v$ is the unique vertex adjacent to every element in $B$. Let $Y$ be a full subgraph of $\Gamma$; the \textit{first rigid expansion} of $Y$, denoted by $Y^{1}$, is the full subgraph spanned by the vertices of $Y$ and all the vertices uniquely determined by sets of vertices of $Y$. The $n$-th rigid expansion is then defined inductively: $Y^{n} = (Y^{n-1})^{1}$. We denote by $Y^{\infty}$ the full subgraph spanned by the union of the vertex sets of $Y^{i}$ for $i \in \nat$. See Section \ref{chap3} below for more details.\\
\indent In this general setting, Theorem \ref{TheoC} below tells us in particular that given a simplicial map from a connected full subgraph $Y$ to $\Gamma$ that coincides with the restriction to $Y$ of an automorphism, the only way to extend it to $Y^{\infty}$ so that the extended map is at least edge-preserving, is via said automorphism of $\Gamma$.
\begin{Theo}\label{TheoC}
 Let $\Gamma$ be a connected simplicial graph, $Y$ be a connected full subgraph of $\Gamma$, and $\fun{\varphi}{\bigcup_{i \in \nat}Y^{i}}{\Gamma}$ be an edge-preserving map such that there exists an automorphism $\phi \in \mathrm{Aut}(\Gamma)$ with $\phi|_{Y} = \varphi|_{Y}$. Then $\varphi = \phi|_{Y^{\infty}}$, and any other $\psi \in \Aut{\Gamma}$ with $\varphi = \psi|_{Y^{\infty}}$ differs from $\phi$ by an element in $\stab{Y}$.
\end{Theo}
\indent Similarly, we can also generalise the concept of a rigid set: we say a full subgraph $Y$ of $\Gamma$ is a \textit{rigid set} if any locally injective map $Y \rightarrow \Gamma$ is the restriction of some automorphism of $\Gamma$. With this definition and Theorem \ref{TheoC} we have the following corollary:
\begin{Coro}\label{CoroD}
 Let $\Gamma$ be a connected simplicial graph, $Y$ be a rigid set of $\Gamma$, and $\fun{\varphi}{Y^{\infty}}{\Gamma}$ be an edge-preserving map such that $\varphi|_{Y}$ is locally injective. Then $\varphi$ is the restriction to $Y^{\infty}$ of an automorphism of $\Gamma$, unique up to the pointwise stabilizer of $Y$ in $\mathrm{Aut}(\Gamma)$.
\end{Coro}
\indent One of the objectives of Theorem \ref{TheoC} and Corollary \ref{CoroD}, is to give a way to obtain new results on the combinatorial rigidity problem of various simplicial graphs (e.g. the pants graph, the Hatcher-Thurston graph, etc.), by one of two ways: Either by finding (suitable) subgraphs for which it can be proved  that the simplicial map is induced by an automorphism of the graph, and proving that the rigid expansion of the subgraph exhaust the original graph (so we can use Theorem \ref{TheoC}); or by finding (preferably finite) rigid sets in them where the restriction of the simplicial map is locally injective, and proving that the rigid expansions of the rigid sets exhaust the graph (so we can use Corollary \ref{CoroD}). Note that due to the abstract setting of the theorem, this \textbf{need not} be done exclusively for simplicial graphs associated to a surface; we hope in the future to use this theorem to find analogous results to those of the curve graph on various complexes associated to other structures, e.g. the various complexes associated to the outer automorphism group of a free group of finite rank.\\
\indent Later on, in Section \ref{prelim} we reintroduce some of the concepts mentioned here and introduce the notation used throughout this article. We also reintroduce the rigid set of \cite{Ara1}, and use it along with Theorem B of \cite{JHH1} and Corollary \ref{CoroD} to obtain an analogous result to Corollary \ref{CoroD} for the curve graph (see Corollary \ref{Thm4}).\\
\indent In Section \ref{chap3sec3} we take advantage of the relation between the curve graph and the topology of the underlying surface, along with the previous corollary (Corollary \ref{Thm4}), to prove Theorem \ref{TheoB}.\\
\indent Finally, we prove an application of this theorem to homomorphisms between subgroups of extended mapping class groups, following Ivanov's recipe in \cite{Ivanov}.
\begin{Coro}\label{IntroCor2}
 Let $S_{1} = S_{g_{1},n_{1}}$ and $S_{2} = S_{g_{2},n_{2}}$, such that $g_{1} \geq 3$ and $\kappa(S_{1}) \geq \kappa(S_{2}) \geq 6$; let also $\Gamma < \EMod{S_{1}}$ be a subgroup such that for every curve $\gamma$ in $S_{1}$ there exists $N \neq 0$ with $\tau_{\gamma}^{N} \in \Gamma$ (where $\tau_{\gamma}$ denotes the right Dehn twist along $\gamma$), and let $\fun{\phi}{\Gamma}{\EMod{S_{2}}}$ be a homomorphism such that:
 \begin{enumerate}
  \item For each curve $\gamma$ in $S_{1}$, there exist $L, M \in \mathbb{Z}^{*}$ such that $\tau_{\gamma}^{L} \in \Gamma$ and $\phi(\tau_{\gamma}^{L}) = \tau_{\delta}^{M}$ for some curve $\delta$ in $S_{2}$.
  \item For any disjoint curves $\alpha$ and $\beta$, there exist $N_{\alpha}, N_{\beta} \neq 0$ such that the subgroup generated by $\phi(\tau_{\alpha}^{N_{\alpha}})$ and $\phi(\tau_{\beta}^{N_{\beta}})$, is not cyclic.
 \end{enumerate}
 Then, $S_{1}$ is homeomorphic to $S_{2}$ and $\phi$ is the restriction to $\Gamma$ of an inner automorphism of $\EMod{S}$ with $S \cong S_{1} \cong S_{2}$. 
\end{Coro}
\indent This corollary is very similar to Corollary 2 in \cite{AraSouto3}. However, it is not clear whether the hypotheses of these two corollaries are equivalent or not.\\
\indent We must remark that this work is the published version of the third chapter of the author's Ph.D. thesis, and the results here presented are dependent on the results found in \cite{JHH1}, which is the published version of the first two chapters. There we prove that using iterated rigid expansions of Aramayona and Leininger's finite rigid set, we can create an increasing sequence of finite rigid sets that exhausts the curve graph. Later on in \cite{JHH3}, the last article of this series, we use Theorem \ref{TheoB} to obtain new results in the combinatorial rigidity of another simplicial graph (the Hatcher-Thurston graph).\\[0.3cm]
\textbf{Acknowledgements:} The author thanks his Ph.D. advisors, Javier Aramayona and Hamish Short, for their very helpful suggestions, talks, corrections, and specially for their patience while giving shape to this work.
%%%%%%%%%%%%%%%%%%%%%%%%%%%%%%%%%%%%%%%%%%%%%%%%%%%%%%%%%%%%
%%%%%%%%%%%%%%%%%%%%%%%%%%%%%%%%%%%%%%%%%%%%%%%%%%%%%%%%%%%%
%%%%%%%%%%%%%%%%%%%%%%%%%%%%%%%%%%%%%%%%%%%%%%%%%%%%%%%%%%%%
\section{Rigid sets and edge-preserving maps}\label{chap3}
\indent In this section we generalise the concepts of a rigid set and rigid expansions first introduced in \cite{Ara2} and then used in \cite{JHH1}. We suppose $\Gamma$ is a simplicial connected graph. Let $Y$ be a \textit{subgraph} of $\Gamma$, denoted by $Y < \Gamma$; we denote its vertex set by $\mathcal{V}(Y)$.\\
\indent Let $v \in \mathcal{V}(\Gamma)$, and $B \subset \mathcal{V}(\Gamma)$. Recall that the link of $v$, denoted by $\lk{v}$, is the full subgraph spanned by all the vertices adjacent to $v$ in $\Gamma$. We say that $v$ is \textit{uniquely determined} by $B$, denoted by $v = \langle B \rangle$, if we have the following: $$\{v\} = \bigcap_{w \in B} \lk{w}.$$
\indent Note that this implies that if $v = \langle B \rangle$ and $\phi \in \Aut{\Gamma}$, we have that $\phi(v) = \langle \phi(B) \rangle$.\\
\indent Let $Y < \Gamma$; the first \textit{rigid expansion} of $Y$, denoted by $Y^{1}$, is the full subgraph spanned by the vertex set:$$\mathcal{V}(Y^{1}) \ColonEqq \mathcal{V}(Y) \cup \{v: v = \langle B \rangle, B \subset \mathcal{V}(Y)\};$$ we also define $Y^{0} = Y$ and, inductively, $Y^{k} = (Y^{k-1})^{1}$.\\
\indent We state some properties of the pointwise stabilizers of a set with respect to the pointwise stabilizer of its rigid expansions.
\begin{Prop}\label{stabilizers1}
 For $Y$ a full subgraph of $\Gamma$, $\stab{Y} = \stab{Y^{1}}$.
\end{Prop}
\begin{proof}
 If $Y^{1} = Y$, then we have the desired result, thus we suppose $Y \subsetneq Y^{1}$. Then $\stab{Y^{1}} \subset \stab{Y}$. Let $\phi \in \stab{Y}$, and $v \in Y^{1} \backslash Y$; as such there exists $C \subset \mathcal{V}(Y)$ with $\beta = \langle C\rangle$. Given that $\phi$ is an automorphism of $\Gamma$, we have that $\phi(v) = \langle \phi(C) \rangle$, and since $\phi \in \stab{Y}$ we get $\phi(v) = \langle C \rangle = v$. Hence $\stab{Y} = \stab{Y^{1}}$.
\end{proof}
\indent If $\{Y^{i}\}_{i \in I}$ is a (possibly finite) sequence of iterated rigid expansions of a full subgraph $Y$ of $\Gamma$, we denote by $Y^{\infty} = \bigcup_{i \in I} Y^{i}$ the full subgraph of $\Gamma$ spanned by the vertex set $\bigcup_{i \in I} \mathcal{V}(Y_{i})$.\\
\indent By induction and following the same argument of Proposition \ref{stabilizers1}, we obtain the following corollary.
\begin{Cor}\label{stabilizers2}
 For $Y$ a full subgraph of $\Gamma$ and every $k \in \nat$, $\stab{Y} = \stab{Y^{k}} = \stab{Y^{\infty}}$
\end{Cor}
\indent Now, we prove that for a restriction of an automorphism to a connected full subgraph of $\Gamma$, there exists a \textbf{unique} (up to $\stab{Y}$) edge-preserving extension to $Y^{1}$. This is the first step to prove Theorem \ref{TheoC}.
\begin{Lema}\label{edgemaplema}
 Let $Y < \Gamma$ be a full subgraph of $\Gamma$. Then any restriction of an automorphism from $Y$ to $\Gamma$ extends uniquely (up to $\stab{Y}$) to an edge-preserving map from $Y^{1}$; i.e. if $\fun{\varphi}{Y^{1}}{\Gamma}$ is an edge-preserving map such that there exists $\phi \in \mathrm{Aut}(\Gamma)$ such that $\phi|_{Y} = \varphi|_{Y}$, then $\varphi = \phi|_{Y^{1}}$, and any other $\psi \in \Aut{\Gamma}$ with $\varphi = \psi|_{Y^{1}}$ differs from $\phi$ by an element in $\stab{Y}$.
\end{Lema}
\begin{proof}
 If $Y^{1} = Y$, then we have the desired result by definition, thus we suppose $Y \subsetneq Y^{1}$. Let $v \in Y^{1} \backslash Y$. As such, there exists $C \subset Y$ such that $v = \langle C\rangle$. This implies that $\phi(v) = \langle \phi(C) \rangle$.\\
 \indent Given that $\varphi$ is an edge-preserving map, $\varphi(v)$ is a vertex in $\Gamma$ adjacent to every element in $\varphi(C) = \phi(C)$. But $\phi(v)$ is uniquely determined by $\phi(C)$, so it is the \textit{only} vertex in $\Gamma$ adjacent to every element in $\phi(C)$. Therefore $\varphi(v) = \phi(v)$. This implies that $\varphi = \phi|_{Y^{1}}$.\\
 \indent We also have that $\phi$ is unique up to $\stab{Y}$ since $\stab{Y} = \stab{Y^{1}}$ by Corollary \ref{stabilizers2}. Therefore $\phi$ is unique up to $\stab{Y}$ as desired.
\end{proof}
\indent Following the same argument as before, we can now prove Theorem \ref{TheoC}.
\begin{proof}[\textbf{Proof of Theorem \ref{TheoC}}] 
 By an iterated use of Lemma \ref{edgemaplema}, we have that for every $k \in \nat$, $\phi|_{Y^{k}} = \varphi|_{Y^{k}}$. If $v \in \mathcal{V}(Y^{\infty})$, then $v \in \mathcal{V}(Y^{k})$ for some $k \in \nat$, thus $\phi(v) = \varphi(v)$.\\
 \indent Therefore, $\varphi$ is the restriction of an automorphism of $\Gamma$, unique up to $\stab{Y^{\infty}} = \stab{Y}$ (by Corollary \ref{stabilizers2}).
\end{proof}
\indent Let $v \in \mathcal{V}(\Gamma)$; the \textit{star of} $v$, denoted $\mathrm{star}(v)$ is defined as the subgraph of $\Gamma$ whose vertex set is $v$ union all the set of vertices adjacent to $v$ in $\Gamma$, and the edges are those edges of $\Gamma$ that have $v$ as one of their endpoints.\\
\indent Let $Y < \Gamma$ be a full subgraph of $\Gamma$. A simplicial map $\fun{\varphi}{Y}{\Gamma}$ is locally injective if for all $v \in \mathcal{V}(Y)$ we have $\varphi|_{\mathrm{star}(v) \cap Y}$ is injective.\\
\indent A \textit{rigid set} $Y < \Gamma$ is a full subgraph such that any locally injective simplicial map $\fun{\varphi}{Y}{\Gamma}$ is the restriction to $Y$ of an automorphism of $\Gamma$, unique up to its pointwise stabilizer $\stab{Y}$ in $\Aut{\Gamma}$.\\
\indent The \textbf{proof of Corollary \ref{CoroD}} follows from the fact that if $Y$ is rigid and $\varphi|_{Y}$ is locally injective, there exists an automorphism $\phi \in \Aut{\Gamma}$ such that $\varphi|_{Y} = \phi|_{Y}$. Then the conditions for Theorem \ref{TheoC} are satisfied.\\[0.3cm]
\indent An immediate consequence of Corollary \ref{CoroD} is the following corollary.
\begin{Cor}
 If $Y < \Gamma$ is a rigid set, $Y^{\infty}$ is rigid.
\end{Cor}
%%%%%%%%%%%%%%%%%%%%%%%%%%%%%%
%%%%%%%%%%%%%%%%%%%%%%%%%%%%%%
\section{The curve graph}\label{prelim}
\indent As stated earlier, suppose $S_{g,n}$ is an orientable surface of finite topological type with genus $g \geq 3$ and $n \geq 0$ punctures. We define the complexity of $S_{g,n}$ as $\kappa(S_{g,n}) \ColonEqq 3g-3+n$. The \textit{extended mapping class of} $S_{g,n}$, denoted by $\EMod{S_{g,n}}$, is the group of isotopy classes of \textit{all} self-homeomorphisms of $S_{g,n}$.\\
\indent A \textit{curve} $\alpha$ is a topological embedding of the unit circle into the surface. We often abuse notation and call ``curve'' the embedding, its image on $S_{g,n}$ or its isotopy class. The context makes clear which use we mean.\\
\indent A curve is \textit{essential} if it is neither null-homotopic nor homotopic to the boundary curve of a neighbourhood of a puncture.\\
\indent The (geometric) intersection number of two (isotopy classes of) curves $\alpha$ and $\beta$ is defined as follows: $$i(\alpha,\beta) \ColonEqq \min \{|a \cap b| : a \in \alpha, b \in \beta\}.$$
\indent Let $\alpha$ and $\beta$ be two curves on $S_{g,n}$. \textbf{Here} we use the convention that $\alpha$ and $\beta$ are \textit{disjoint} if $i(\alpha,\beta) = 0$ \textbf{and} $\alpha \neq \beta$.\\
\indent For $\kappa(S_{g,n}) > 1$, the curve graph of $S_{g,n}$, denoted by $\ccomp{S_{g,n}}$, is the simplicial graph whose vertices are the isotopy classes of essential curves on $S_{g,n}$, and two vertices span an edge if the corresponding curves are disjoint.\\
\indent For $\kappa(S_{g,n}) = 1$, $\ccomp{S_{g,n}}$ is the simplicial graph whose vertices are the isotopy classes of essential curves on $S_{g,n}$, and two vertices span an edge if the corresponding curves intersect minimally (intersection $1$ for $S_{1,1}$ and intersection $2$ for $S_{0,4}$). Note that in this case, $\ccomp{S_{g,n}}$ is the $1$-skeleton of the Farey complex, called the Farey graph. The Farey complex can be thought of as an ideal triangulation of the Poincar\'{e} disc-model of the hyperbolic plane.\\
\indent A \textit{multicurve} $M$ is a set of pairwise disjoint curves. This implies that in the curve graph, the full subgraph spanned by $M$ is a complete subgraph. A \textit{pants decomposition} $P$ of $S_{1}$ is a maximal multicurve, i.e. it is a maximal complete subgraph of $\ccomp{S_{1}}$. Note that $|P| = \kappa(S_{1})$.\\
\indent Now, we reintroduce the finite rigid sets of \cite{Ara1}.
\subsection{$\X(S)$ for closed surfaces}\label{chap3sec1subsec1}
\indent In this subsection we suppose $S = S_{g,0}$. Let $k \in \mathbb{Z}^{+}$ and $C = \{\gamma_{0}, \ldots, \gamma_{k}\}$ be an ordered set of $k+1$ curves in $S$. It is called a \textit{chain} of length $k+1$ if $i(\gamma_{i},\gamma_{i+1}) = 1$ for $0 \leq i \leq k-1$, and $\gamma_{i}$ is disjoint from $\gamma_{j}$ for $|i - j| > 1$. On the other hand, $C$ is called a \textit{closed chain} of length $k+1$ if $i(\gamma_{i},\gamma_{i+1}) = 1$ for $0 \leq i \leq k$ modulo $k+1$, and $\gamma_{i}$ is disjoint from $\gamma_{j}$ for $|i - j| > 1$ (modulo $k+1$); a closed chain is called \textit{maximal} if it has length $2g+2$.\\
\indent A \textit{subchain} is an ordered subset of either a chain or a closed chain which is itself a chain, and its length is its cardinality. A \textit{bounding pair} associated to a subchain $C$ of odd length, is the pair $\{\gamma_{1}, \gamma_{2}\}$ of boundary curves of the regular neighbourhood of $C$.\\
\indent Let $\Cf = \{\alpha_{0}, \ldots, \alpha_{2g+1}\}$ be the closed chain in $S$ depicted in Figure \ref{OriginalChainv2}.
\begin{figure}[h]
\begin{center}
 \includegraphics[width=9cm]{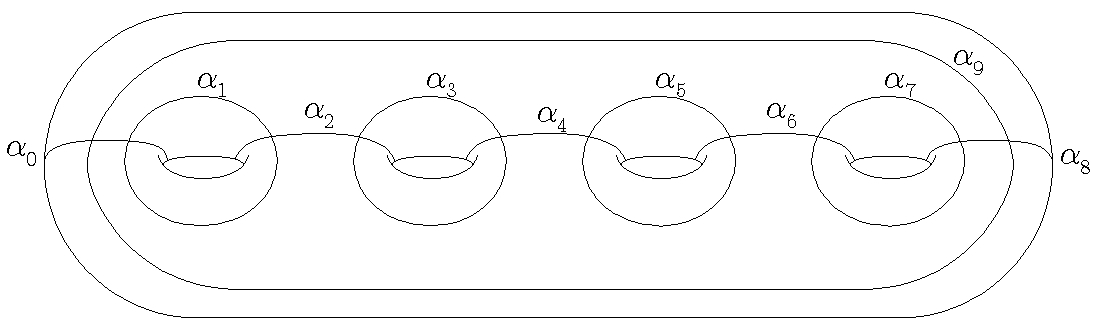} \caption{The set $\Cf = \{\alpha_{0}, \ldots, \alpha_{2g+1}\}$.}\label{OriginalChainv2} 
\end{center}
\end{figure}\\
\indent Note that if we cut the surface along the curves $\{\alpha_{2k}\}_{k = 0}^{g}$ we separate the surface in two connected components each homeomorphic to $S_{0,g+1}$. We fix the notation of $S_{e}^{+}$ and $S_{e}^{-}$ to the subsurfaces of $S$ corresponding to these connected components; see Figure \ref{Beven} for an example. Analogously, using the set $\{\alpha_{2k+1}\}_{k=0}^{g}$ we separate the surface in two connected components each homeomorphic to $S_{0,g+1}$, and we fix the notation of $S_{o}^{+}$ and $S_{o}^{-}$ to the subsurfaces of $S$ corresponding to these connected components.
\begin{figure}[h]
\begin{center}
 \includegraphics[width=9cm]{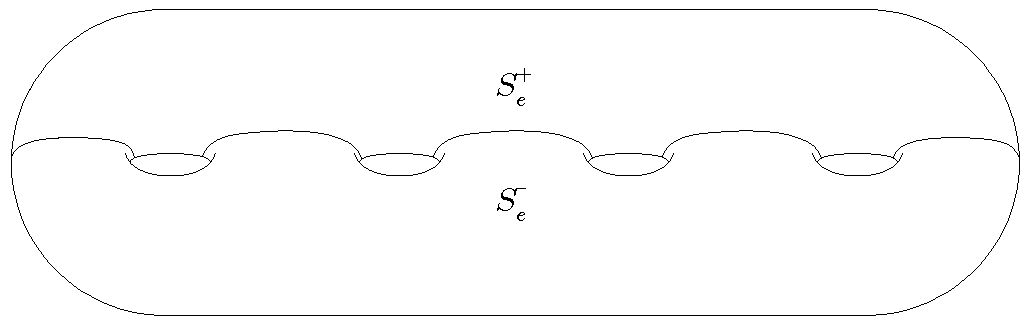} \caption{A convention of $S_{e}^{+}$ and $S_{e}^{-}$.}\label{Beven}
\end{center}
\end{figure}\\
\indent Let $J$ be a subinterval (modulo $2g+2$) of $\{0, \ldots, 2g+1\}$ such that $|J| < 2g -1$.\\
\indent If $J = \{j, \ldots, j+2k\}$ for some $j \in \nat$, we denote by $\beta_{J}^{+}$ and $\beta_{J}^{-}$ the elements of the bounding pair associated to the chain $\{\alpha_{j}, \ldots, \alpha_{j+2k}\}$, according to whether they are contained in either $S_{e}^{+}$ (or $S_{o}^{+}$) or $S_{e}^{-}$ (or $S_{o}^{-}$).\\
\indent This way, we define the set: $$\Bf \ColonEqq \{\beta_{J}^{\pm}: |J| = 2k+1 \hspace{0.2cm} \mathrm{for} \hspace{0.2cm} \mathrm{some} \hspace{0.2cm} k \in \mathbb{Z}^{+}\}.$$
\indent If $J = \{j, \ldots, j+ 2k-1\}$ for some $j \in \nat$ and $k \in \mathbb{Z}^{+}$, we get the following curve: $$\sigma_{J} \ColonEqq \langle \{\alpha_{j}, \ldots, \alpha_{j+2k-1}\} \cup \{\alpha_{j + 2k +1}, \ldots, \alpha_{j-2}\}\rangle.$$ For examples see Figure \ref{ExamplesSigmaSec3fig1}. 
\begin{figure}[h]
 \begin{center}
  \resizebox{10cm}{!}{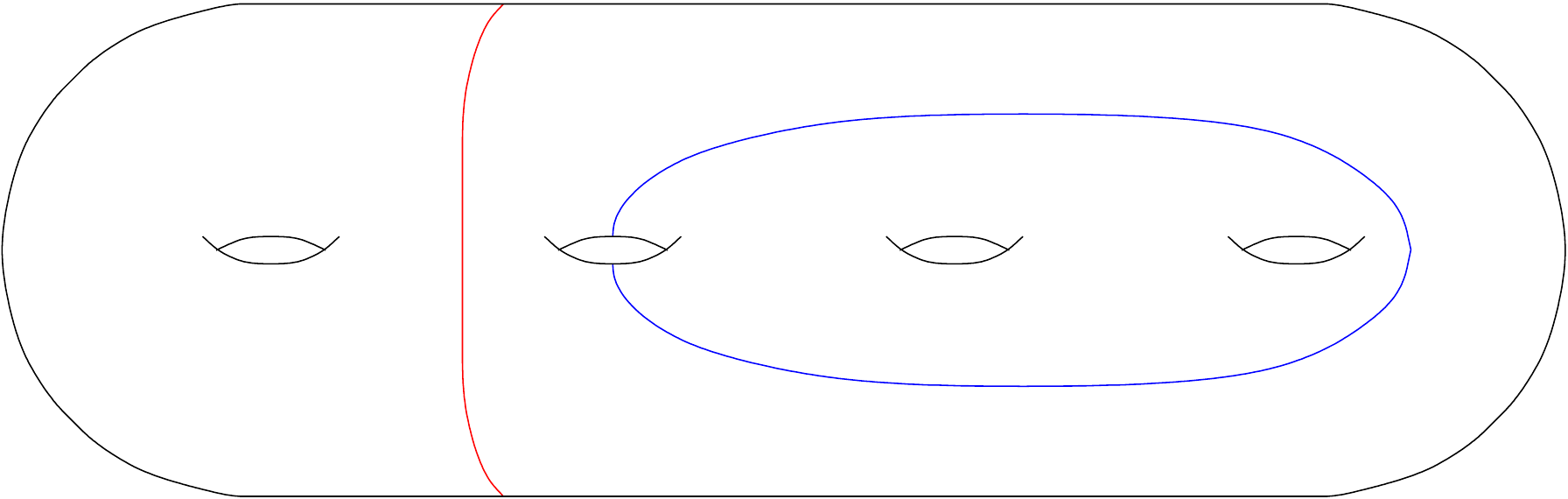} \caption{Examples of $\sigma_{\{0,1\}}$ in red, and $\sigma_{\{4,5,6,7\}}$ in blue.}\label{ExamplesSigmaSec3fig1}
 \end{center}
\end{figure}\\
This way, we define the following set: $$\Sf \ColonEqq \{\sigma_{J} : |J| = 2k \hspace{0.2cm} \mathrm{for} \hspace{0.2cm} \mathrm{some} \hspace{0.2cm} k \in \mathbb{Z}^{+}\}.$$
\indent If $J = \{j, \ldots, j + 2k\}$ for some $j \in \nat$ and $k \in \mathbb{Z}^{+}$, let we get the following curves:$$\mu_{j+2k+1, J}^{+} \ColonEqq \langle \{\beta_{J}^{+},\alpha_{j+2k+1}\} \cup \{\alpha_{j}, \ldots, \alpha_{j+2k-1}\} \cup \{\beta_{\{j+2, \ldots, j+2k+2\}}^{-}\} \cup \{\alpha_{j+2k+3}, \ldots, \alpha_{j-2}\} \rangle,$$ $$\mu_{j+2k+1, J}^{-} \ColonEqq \langle \{\beta_{J}^{-},\alpha_{j+2k+1}\} \cup \{\alpha_{j}, \ldots, \alpha_{j+2k-1}\} \cup \{\beta_{\{j+2, \ldots, j+2k+2\}}^{+}\} \cup \{\alpha_{j+2k+3}, \ldots, \alpha_{j-2}\} \rangle.$$
\indent For examples see Figure \ref{ExamplesMuSec3fig1}. 
\begin{figure}[h]
 \begin{center}
  \resizebox{10cm}{!}{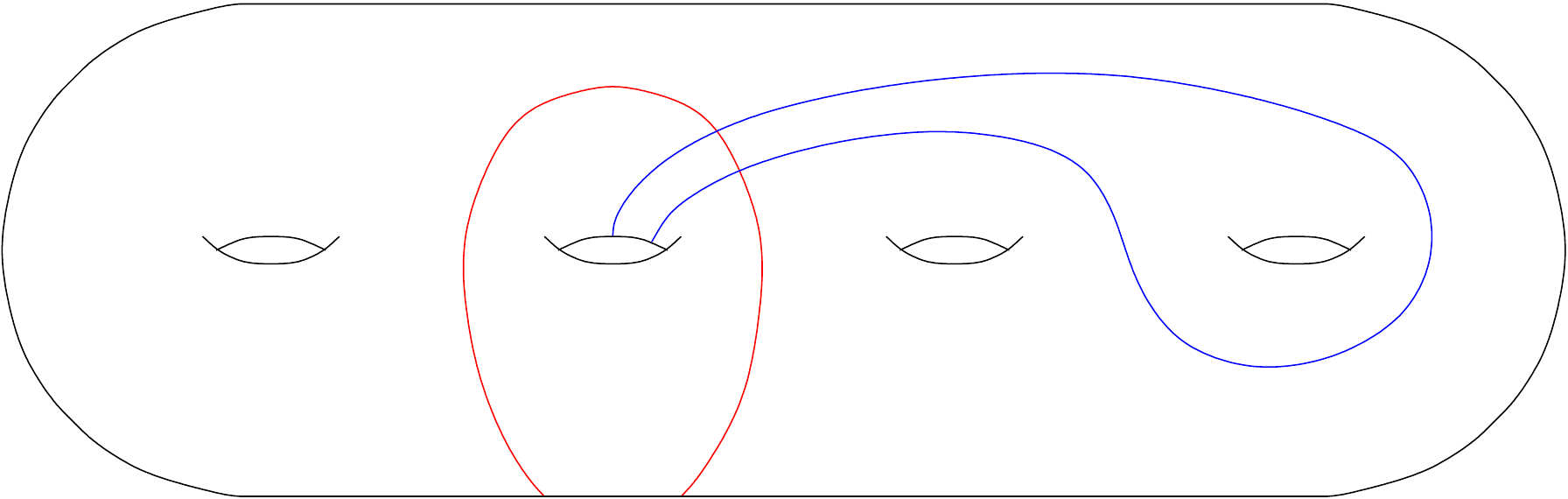} \caption{Examples of $\mu_{3,\{0,1,2\}}$ in red, and $\sigma_{7,\{4,5,6\}}$ in blue.}\label{ExamplesMuSec3fig1}
 \end{center}
\end{figure}\\
\indent This way, we define the following set: $$\Mf \ColonEqq \{\mu_{i,J}^{\pm}: J = \{j, \ldots, j+2k\}  \hspace{0.2cm} \mathrm{for} \hspace{0.2cm} \mathrm{some} \hspace{0.2cm} k \in \mathbb{Z}^{+}, i = j+2k+1\}$$
\indent Finally, we have the set $$\X(S) \ColonEqq \Cf \cup \Bf \cup \Sf \cup \Mf$$
\indent Recall that, as was mentioned before, this set was proved to be rigid in \cite{Ara1}, and by construction has trivial pointwise stabilizer in $\EMod{S}$.
%%%%%%%%%%%
\subsection{$\X(S)$ for punctured surfaces}\label{chap3sec1subsec2}
\indent In this subsection we suppose $S = S_{g,n}$ with $n \geq 1$. Let $\Cf_{0} = \{\alpha_{1}, \ldots, \alpha_{2g+1}\}$ be the chain depicted in Figure \ref{OriginalChainPunctured}, and $\Cf_{f} = \{\alpha_{0}^{0}, \alpha_{0}^{1}, \ldots, \alpha_{0}^{n}\}$ be the multicurve also depicted in Figure \ref{OriginalChainPunctured}. Now, let $\Cf \ColonEqq \Cfn \cup \Cff$.
\begin{figure}[h]
 \begin{center}
  \resizebox{10cm}{!}{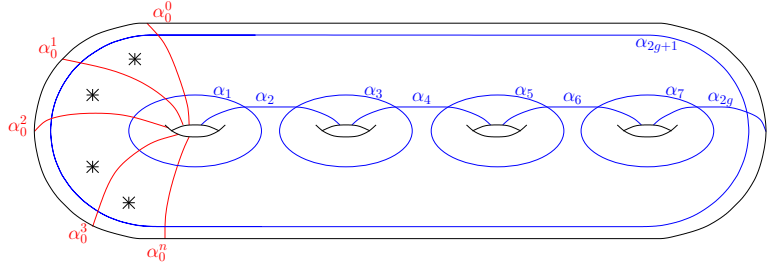} \caption{$\Cf_{0}$ in blue and $\Cf_{f}$ in red for $S_{5,4}$.}\label{OriginalChainPunctured}
 \end{center}
\end{figure}\\
\indent For $i \in \{0, \ldots, n\}$, let us consider the closed chain $C_{i} = \{\alpha_{0}^{i}, \alpha_{1}, \ldots, \alpha_{2g+1}\}$. Then we denote the curve $\alpha_{0}^{i}$ by $\alpha_{0}$ to simplify notation when it is understood that $\alpha_{0} \in C_{i}$. As such $C_{i}$ has the subsets: $C_{i(o)} = \{\alpha_{j} \in C_{i} : j$ is odd$\}$ and $C_{i(e)} = \{\alpha_{j} \in C_{i} : j$ is even$\}$. These subsets are such that:
\begin{itemize}
 \item $S \backslash C_{i(e)}$ has two connected components, $S_{i(e)}^{+} = S_{0,i+g+1}$ and $S_{i(e)}^{-} = S_{0, n-i+ g+1}$.
 \item $S \backslash C_{i(o)}$ has two connected components, $S_{i(o)}^{+} = S_{0,n+g+1}$ and $S_{i(o)}^{-} = S_{0,g+1}$.
\end{itemize}
\indent Recalling that the subindices are modulo $2g+2$, we denote by $[\alpha_{j}, \ldots, \alpha_{j + 2k}]^{+}$ for some $0 < k < g-1$, the boundary component of a closed regular neighbourhood $N(\{\alpha_{j}, \ldots, \alpha_{j + 2k}\})$, that is contained in either $S_{i(o)}^{+}$ or in $S_{i(e)}^{+}$. Analogously, for $[\alpha_{j}, \ldots, \alpha_{j + 2k}]^{-}$ we denote the boundary component of a closed regular neighbourhood $N(\{\alpha_{j}, \ldots, \alpha_{j + 2k}\})$, that is contained in either $S_{i(o)}^{-}$ or in $S_{i(e)}^{-}$.\\
\indent Let $J = \{2l, \ldots, 2(l+k)\}$, for some $1 \leq k \leq g-1$, be a proper interval in the cyclic order modulo $2g+2$. Let also $\beta_{J}^{\pm} = [\alpha_{2l}, \ldots, \alpha_{2(l+k)}]^{\pm}$ (with $\alpha_{0} = \alpha_{0}^{1}$ if necessary). See Figure \ref{ExamplesBf0} for examples. We define $$\Bf_{0} \ColonEqq \{\beta_{J}^{\pm} : J = \{2l, \ldots, 2(l+k)\}, 1 \leq k \leq g-1\}.$$
\begin{figure}[h]
 \begin{center}
  \resizebox{10cm}{!}{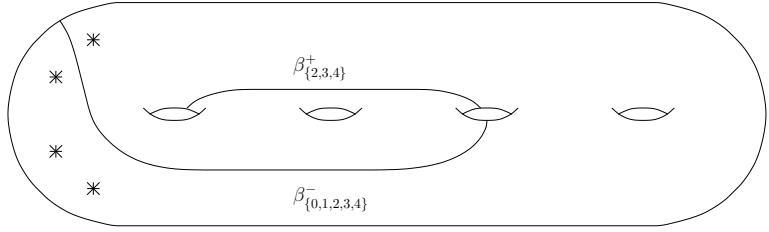}\caption{Examples of curves $\beta_{\{2,3,4\}}^{+}$ and $\beta_{\{0,1,2,3,4\}}^{-}$.}\label{ExamplesBf0}
 \end{center}
\end{figure}\\
\indent For $0 \leq i \leq n -2$, we define $$\eps{i}{i+2} \ColonEqq \langle\Cfn \cup (\Cff \backslash \{\alpha_{0}^{i+1}\})\rangle;$$ note that $\eps{i}{i+2} \in \Cf^{1}$; this can be seen using Figure \ref{OriginalChainPunctured} and removing $\alpha_{0}^{i+1}$ for the chosen $i$.\\
\indent For $0 \leq i < j \leq n$ with $j - i > 2$, we define the curve: $$\eps{i}{j} \ColonEqq \langle\Cfn \cup (\Cff \backslash \{\alpha_{0}^{k} : i < k < j\}) \cup \{\eps{k}{k+2} : i \leq k \leq j-2\}\rangle;$$ note that $\eps{i}{j} \in \Cf^{2}$, and that $\eps{i}{j}$ with $j-i > 1$ is the boundary curve of a disc in $S$ containing $j-i$ punctures.\\
\indent Then, we define the set: $$\Df \ColonEqq \{\eps{i}{j} : j-i > 1\}.$$
\indent Now, let $0 \leq i \leq j \leq n$ and consider the following closed regular neighbourhoods:
\begin{center}
 \begin{tabular}{cc}
  $N_{2}^{i,j} = N(\{\alpha_{0}^{i},\alpha_{0}^{j},\alpha_{1}, \alpha_{2}\})$ & $N_{2g}^{i,j} = N(\{\alpha_{0}^{i}, \alpha_{0}^{j}, \alpha_{2g+1}, \alpha_{2g}\})$
 \end{tabular}
\end{center}
\indent Note that both $N_{2}^{i,j}$ and $N_{2g}^{i,j}$ have interiors homeomorphic to $S_{1,3}$, and each has $\eps{i}{j}$ as one of its boundary curves when $j-i > 1$.\\
\indent In $N_{2}^{i,j}$ we define $\beta_{\{0,1,2\}}^{i,+}$ as the boundary curve of $N_{2}^{i,j}$ contained in $S_{i(e)}^{+}$, and $\beta_{\{0,1,2\}}^{j,-}$ as the boundary curve contained in $S_{j(e)}^{-}$. See Figure \ref{ExamplesBetaT1} for examples.
\begin{figure}[h]
 \begin{center}
  \resizebox{10cm}{!}{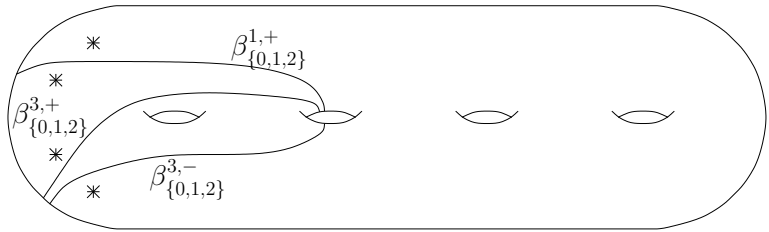} \caption{Examples of curves $\beta_{\{0,1,2\}}^{1,+}$, $\beta_{\{0,1,2\}}^{3,+}$ and $\beta_{\{0,1,2\}}^{3,-}$ in $S_{4,4}$.}\label{ExamplesBetaT1}
 \end{center}
\end{figure}\\
\indent Analogously for $N_{2g}^{i,j}$, we define $\beta_{\{2g,2g+1,0\}}^{i,+}$ as the boundary curve of $N_{2g}^{i,j}$ contained in $S_{i(e)}^{+}$, and $\beta_{\{2g,2g+1,0\}}^{j,-}$ as the boundary curve contained in $S_{j(e)}^{-}$. See Figure \ref{ExamplesBetaT2} for examples.
\begin{figure}[h]
 \begin{center}
  \resizebox{10cm}{!}{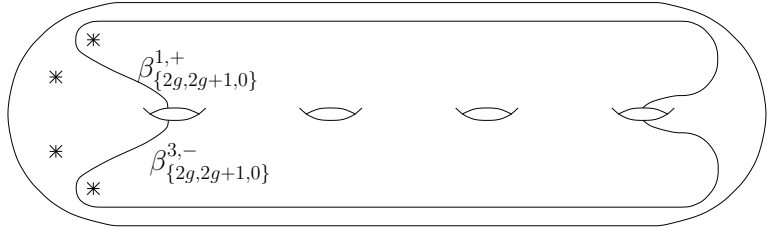} \caption{Examples of curves $\beta_{\{2g,2g+1,0\}}^{1,+}$ and $\beta_{\{2g,2g+1,0\}}^{3,-}$ in $S_{4,4}$.}\label{ExamplesBetaT2}
 \end{center}
\end{figure}\\
\indent Then, we define the set $$\Bf_{T} \ColonEqq \{\beta_{J}^{i, +}, \beta_{J}^{i,-} : J \in \{ \{0,1,2\},\{2g,2g+1,0\}\}, 1 \leq i \leq n\}.$$
\indent For $0 \leq i \leq j \leq n$, we denote by $N_{1}^{i, \hspace{0.05cm} j}$ and $N_{2g+1}^{i, \hspace{0.05cm} j}$ closed regular neighbourhoods of the chains $\{\alpha_{0}^{i}, \alpha_{0}^{j}, \alpha_{1}\}$ and $\{\alpha_{0}^{i}, \alpha_{0}^{j}, \alpha_{2g+1}\}$ respectively.\\
\indent Note that $\N{1}{i}{j}$ is a two-holed torus if $j - i \geq 1$ (one of the boundary components will not be essential if $j - i = 1$). Also, $S \backslash \N{1}{i}{j}$ is the disjoint union of a subsurface homeomorphic to an at least once-punctured open disc, and a subsurface homeomorphic to $S_{g-1,n-(j-1)+1}$.\\
\indent If $j - i > 1$, one of the boundary components of $\N{1}{i}{j}$ is the curve $\eps{i}{j}$. On the other hand, for $0 \leq i \leq j \leq n$, we denote by $\sigma_{1}^{i, \hspace{0.05cm} j}$ the boundary component of $\N{1}{i}{j}$ such that one of the connected components of $S \backslash \{\sigma_{1}^{i, \hspace{0.03cm} j}\}$ is homeomorphic to $S_{1,j-i + 1}$.\\
\indent We denote by $\sigma_{2g+1}^{i, \hspace{0.05cm} j}$ the analogous boundary curves of $\N{2g+1}{i}{j}$ (whenever they are essential).\\
\indent Then, we define $$\Sf_{T} \ColonEqq \{\sigma_{l}^{i, \hspace{0.05cm} j} : l \in \{1,2g+1\}, 0 \leq i \leq j \leq n\}$$
\indent Now, let $J$ be a subinterval of $\{0, \ldots, 2g+1\}$ (modulo $2g+2$) such that $|J| \leq 2g$.\\
\indent If $J = \{i, \ldots, i + 2k-1\}$ for some $k \in \mathbb{Z}^{+}$, let $\sigma_{J} = [\alpha_{i}, \ldots, \alpha_{i+2k-1}]$ (with $\alpha_{0} = \alpha_{0}^{1}$ if necessary). We define $$\Sf_{0} \ColonEqq \{\sigma_{J} : J = \{i, \ldots, i + 2k-1\}, k \in \mathbb{Z}^{+}\}.$$
\indent If $J = \{2l, \ldots, 2(l+k)\}$, for some $k \in \mathbb{Z}^{+}$, and $j = 2(l+k) +1$, then $i(\alpha_{j},\beta_{J}^{+}) = i(\alpha_{j},\beta_{J}^{-}) = 1$. Let $\mu_{j,J}^{+}$ be the boundary curve of a regular neighbourhood of $\{\alpha_{j},\beta_{J}^{+}\}$. Analogously, let $\mu_{j,J}^{-}$ be the boundary curve of a regular neighbourhood of $\{\alpha_{j},\beta_{J}^{-}\}$. We define $$\Mf \ColonEqq \{\mu_{j,J}^{\pm} : J = \{2l, \ldots, 2(l+k)\}, k \in \mathbb{Z}^{+}, j = 2(l+k) +1\}.$$
\indent Therefore, we define: $$\X \ColonEqq \Cf \cup \Df \cup \Sf_{T} \cup \Sf_{0} \cup \Bf_{T} \cup \Bf_{0} \cup \Mf.$$
\indent Recall that, as was mentioned before, this set was proved to be rigid in \cite{Ara1}, and by construction has trivial pointwise stabilizer in $\EMod{S}$.
%%%%%%%%%%%
\subsection{Consequences of Corollary \ref{CoroD} in $\ccomp{S}$}\label{chap3sec2subsec3}
\indent The set $\X(S)$ is studied in \cite{Ara1} and \cite{Ara2}, and it is proven to be a finite rigid set of $\ccomp{S}$ (Theorems 5.1 and 6.1 in \cite{Ara1}). Also, in \cite{JHH1} we have the following result.
\begin{Teono}[B in \cite{JHH1}]
 Let $S$ be an orientable surface of genus $g \geq 3$, $n \geq 0$ punctures and empty boundary. Then $\X(S)^{\infty} = \ccomp{S}$.
\end{Teono}
\indent Using this and Corollary \ref{CoroD} we get the following corollary.
\begin{Cor}\label{Thm4}
 Let $S$ be an orientable surface of genus $g \geq 3$, $n \geq 0$ punctures and empty boundary, and $\fun{\varphi}{\ccomp{S}}{\ccomp{S}}$ be an edge-preserving map such that $\varphi|_{\X(S)}$ is locally injective. Then $\varphi$ is induced by a (unique) homeomorphism.
\end{Cor}
\indent As we prove in the following section, this can be generalised even further.
%%%%%%%%%%%%%%%%%%%%%%%%%%%%%%%%%%%%%%%%%%%%%%%%%%%%%%%%%%%%
%%%%%%%%%%%%%%%%%%%%%%%%%%%%%%%%%%%%%%%%%%%%%%%%%%%%%%%%%%%%
%%%%%%%%%%%%%%%%%%%%%%%%%%%%%%%%%%%%%%%%%%%%%%%%%%%%%%%%%%%%
\section{Edge-preserving maps}\label{chap3sec3}
\indent This section is organized as follows: In Subsection \ref{chap3sec3subsec1} we first give some definitions to create a ``generalisation'' of superinjectivity (recall this means that curves that intersect are mapped to curves that intersect) and prove that an edge-preserving map $\fun{\varphi}{\ccomp{S_{1}}}{\ccomp{S_{2}}}$ with $\kappa(S_{1}) \geq \kappa(S_{2})$ satisfies this generalisation. In Subsection \ref{chap3sec3subsec2} we obtain enough topological data from $\varphi$ to prove that $S_{1}$ is homeomorphic to $S_{2}$. Finally, we prove that it is possible to apply Corollary \ref{Thm4} so that $\varphi$ is induced by a homeomorphism. Note that many of the proofs in these subsections are inspired either partially or in spirit by Shackleton's work in \cite{Shack}.
%%%%%%%%%%%%
\subsection{Farey maps}\label{chap3sec3subsec1}
\indent We say a set of curves $A \subset \ccomp{S}$ fills a subsurface $N \subset S$ if $N \backslash A$ is the disjoint union of punctured closed discs (if $N$ has nonempty boundary), open discs, and open punctured discs.\\ 
\indent Let $\alpha, \beta \in \ccomp{S_{1}}$. We say $\alpha$ and $\beta$ are \textit{Farey neighbours} if they fill a subsurface $N \subset S_{1}$ of complexity one and either $i(\alpha,\beta)=1$ (if $N$ has positive genus) or $i(\alpha,\beta)=2$ (if $N$ has genus zero). Note that this means that $\alpha$ and $\beta$ are adjacent vertices in $\ccomp{N}$, since $N$ has complexity one.\\
\indent Let $\alpha$ and $\beta$ be Farey neighbours and $N$ be a regular neighbourhood of $\{\alpha,\beta\}$; we say they are \textit{toroidal-Farey neighbours} if $N$ has genus $1$ and we say they are \textit{spherical-Farey neighbours} if $N$ has genus $0$. See Figure \ref{ExamplesFareyNeighboursSec3} for an example.
\begin{figure}[h]
 \begin{center}
  \resizebox{10cm}{!}{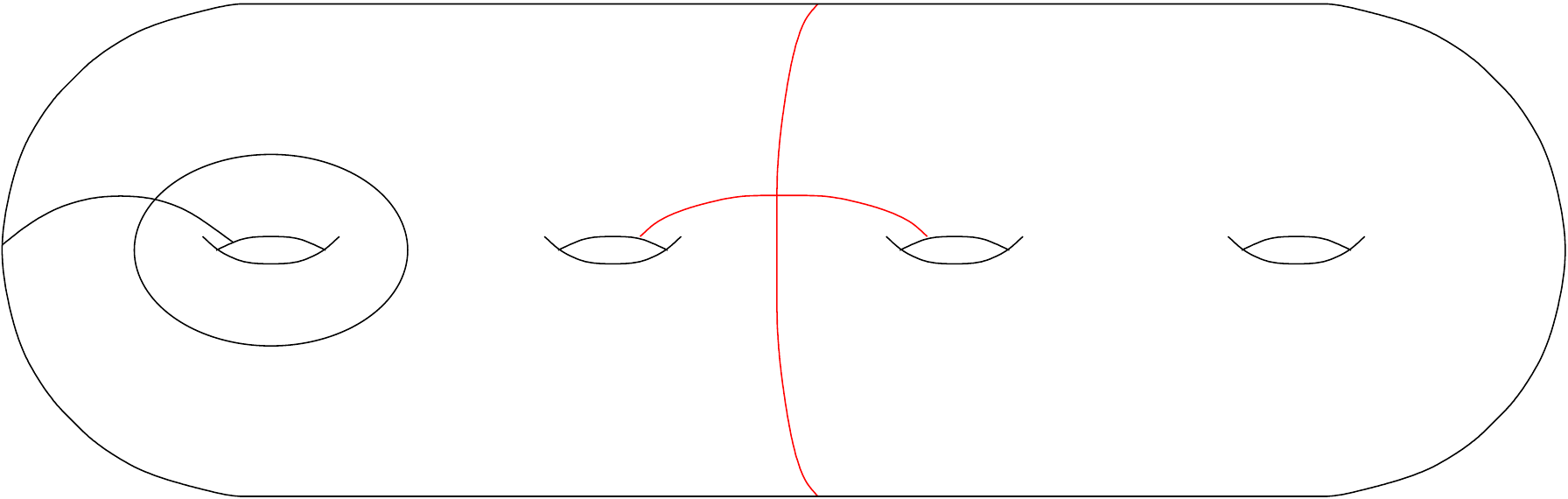} \caption{Examples of toroidal-Farey neighbours in black, and spherical-Farey neighbours in red.}\label{ExamplesFareyNeighboursSec3}
 \end{center}
\end{figure}\\
\indent Let $\fun{\varphi}{\ccomp{S_{1}}}{\ccomp{S_{2}}}$ be a simplicial map. We say $\varphi$ is a \textit{toroidal-Farey map} if for every pair of toroidal-Farey neighbours $\alpha$ and $\beta$, we have that $i(\varphi(\alpha),\varphi(\beta)) \neq 0$. On the other hand, we say $\varphi$ is a \textit{spherical-Farey map} if for every pair of spherical-Farey neighbours $\alpha$ and $\beta$, we have that $i(\varphi(\alpha),\varphi(\beta)) \neq 0$. Finally, we say $\varphi$ is a \textit{Farey map} if it is both toroidal-Farey and spherical-Farey. Note that a (toroidal, or spherical, or both) Farey map is a generalisation of a superinjective map.\\
\indent Now we give some definitions and prove several technical lemmas to prove that if $\varphi$ is an edge-preserving map, then it is also a Farey map.
\begin{Lema}[cf. Lemma 5 in \cite{Shack}]\label{kS1equalkS2}
 Let $S_{1} = S_{g_{1},n_{1}}$ and $S_{2} = S_{g_{2},n_{2}}$ be orientable surface of finite topological type, with $g_{1} \geq 3$, empty boundary, and $n_{1},n_{2} \geq 0$ punctures; let also $\fun{\varphi}{\ccomp{S_{1}}}{\ccomp{S_{2}}}$ be an edge-preserving map. Then $\varphi$ maps multicurves on $S_{1}$ to multicurves of the same cardinality on $S_{2}$. In particular $\kappa(S_{2}) \geq \kappa(S_{1})$.
\end{Lema}
\begin{proof}
 \indent Let $M$ be a multicurve of cardinality $m > 1$. Since $\varphi$ is edge-preserving, then complete subgraphs are mapped to complete subgraphs with vertex sets of the same cardinality. This implies that $\varphi(M)$ is a multicurve, and $\varphi(M)$ has the same cardinality as $M$.\\
 \indent In particular if $P$ is a pants decomposition, then $\varphi(P)$ is a multicurve of cardinality $\kappa(S_{1})$, thus $\kappa(S_{2})$ is at least $\kappa(S_{1})$.
\end{proof}
\indent Armed with this lemma, if $\kappa(S_{1}) \geq \kappa(S_{2})$ then $S_{1}$ and $S_{2}$ must have the same complexity. This in particular gives the following corollary.
\begin{Cor}\label{corkS1equalkS2}
 Let $S_{1} = S_{g_{1},n_{1}}$ and $S_{2} = S_{g_{2},n_{2}}$ be orientable surface of finite topological type, with $g_{1} \geq 3$, empty boundary, and $n_{1},n_{2} \geq 0$ punctures; let also $\fun{\varphi}{\ccomp{S_{1}}}{\ccomp{S_{2}}}$ be an edge-preserving map. Then, $\varphi$ maps pants decompositions to pants decompositions if and only if $\kappa(S_{1}) \geq \kappa(S_{2})$.
\end{Cor}
\begin{Rem}\label{kS1geq6}
 Note that if we assume $g_{1} \geq 3$ then $\kappa(S_{1}) \geq 6$; if we also assume $\kappa(S_{1}) \geq \kappa(S_{2})$, by Lemma \ref{kS1equalkS2} we have $\kappa(S_{2}) = \kappa(S_{1}) \geq 6$.
\end{Rem}
\indent We prove now that any edge-preserving map is a toroidal-Farey map if we add the complexity condition mentioned above.
\begin{Lema}\label{EdgePresToroidal}
 Let $S_{1} = S_{g_{1},n_{1}}$ and $S_{2} = S_{g_{2},n_{2}}$ be orientable surfaces of finite topological type, with $g_{1} \geq 3$, empty boundary, and $n_{1},n_{2} \geq 0$ punctures, such that $\kappa(S_{1}) \geq \kappa(S_{2})$. If $\fun{\varphi}{\ccomp{S_{1}}}{\ccomp{S_{2}}}$ is an edge-preserving map, we have that $\varphi$ is a toroidal-Farey map.
\end{Lema}
\begin{proof}
 Let $\alpha$ and $\beta$ be any two toroidal-Farey neighbours, we need to prove that $i(\varphi(\alpha),\varphi(\beta)) \neq 0$.\\
 \indent Let $M$ be a multicurve on $S_{1}$ such that $\{\alpha\} \cup M$ and $\{\beta\} \cup M$ are pants decompositions. Then by Corollary \ref{corkS1equalkS2} $\varphi(\{\alpha\} \cup M)$ and $\varphi(\{\beta\} \cup M)$ are pants decompositions, which implies that $\varphi(\alpha)$ and $\varphi(\beta)$ are disjoint from every element in $\varphi(M)$. Thus, there exists a complexity-one subsurface of $S_{2}$ containing as essential curves both $\varphi(\alpha)$ and $\varphi(\beta)$. So, either $\varphi(\alpha) = \varphi(\beta)$ or $i(\varphi(\alpha),\varphi(\beta)) \neq 0$.\\
 \indent To prove that $\varphi(\alpha) \neq \varphi(\beta)$, let $\gamma \in \ccomp{S_{1}}$ be such that $i(\beta,\gamma) = 1$ and $\alpha$ is disjoint from $\gamma$. See Figure \ref{ExampleEdgeisFarey} for an example.
 \begin{figure}[h]
  \begin{center}
   \resizebox{10cm}{!}{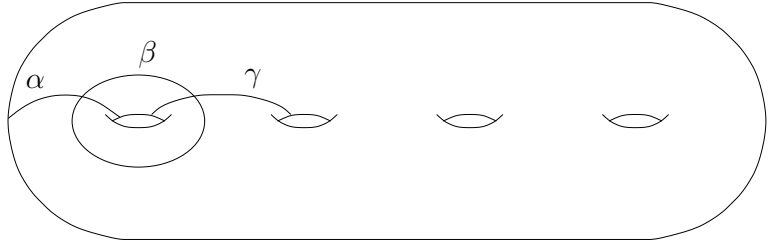} \caption{Examples of curves $\alpha$, $\beta$ and $\gamma$, with $\alpha$ and $\gamma$ disjoint, and $\beta$ and $\gamma$ toroidal-Farey neighbours.}\label{ExampleEdgeisFarey}
  \end{center}
 \end{figure}\\
 \indent Hence $\varphi(\alpha)$ is disjoint from $\varphi(\gamma)$ and, by the same arguments as above, either $\varphi(\beta) = \varphi(\gamma)$ or $i(\varphi(\beta),\varphi(\gamma)) \neq 0$. Neither of these options can happen if $\varphi(\alpha) = \varphi(\beta)$, therefore $\varphi(\alpha) \neq \varphi(\beta)$, which implies that $i(\varphi(\alpha), \varphi(\beta)) \neq 0$, as desired.
\end{proof}
\indent Finally, this lemma implies that to prove that an edge-preserving map under the complexity conditions used above is a Farey map, we only need to prove now that it is a spherical-Farey map. This is done in the following lemma, but first we give a brief (and technical) definition used in the proof.\\
\indent Let $\alpha$ and $\beta$ be two curves on $S_{1}$ which are spherical Farey neighbours with a closed regular neighbourhood $N$, and let $\gamma$ and $\gamma^{\prime}$ be two boundary curves of $N$; we say $\gamma$ and $\gamma^{\prime}$ are \textit{connected outside} $N$ if there exists a proper arc in $S_{1} \backslash \mathrm{int}(N)$ with one endpoint in $\gamma$ and the other in $\gamma^{\prime}$.
\begin{Lema}\label{toroidalSpherical}
 Let $S_{1} = S_{g_{1},n_{1}}$ and $S_{2} = S_{g_{2},n_{2}}$ be orientable surface of finite topological type, with $g_{1} \geq 3$, empty boundary, and $n_{1},n_{2} \geq 0$ punctures, such that $\kappa(S_{1}) \geq \kappa(S_{2})$. If $\fun{\varphi}{\ccomp{S_{1}}}{\ccomp{S_{2}}}$ is an edge-preserving map, then $\varphi$ is a Farey map.
\end{Lema}
\begin{proof}
 By Lemma \ref{EdgePresToroidal}, it suffices to prove that $\varphi$ is also a spherical-Farey map.\\
 \indent Let $\alpha,\beta \in \ccomp{S_{1}}$ be spherical Farey neighbours. We proceed as in Lemma \ref{EdgePresToroidal}. Let $M$ be a multicurve of $\kappa(S_{1}) -1$ elements, such that $M \cup \{\alpha\}$ and $M \cup \{\beta\}$ are pants decompositions. By Lemma \ref{kS1equalkS2} we know that $\varphi(M \cup \{\alpha\})$ and $\varphi(M \cup \{\beta\})$ are also pants decompositions, thus $\varphi(\alpha)$ and $\varphi(\beta)$ are contained in a complexity-one subsurface of $S_{2}$. We then only need to prove that $\varphi(\alpha) \neq \varphi(\beta)$.\\
 \indent Let $N$ be a closed regular neighbourhood of $\alpha$ and $\beta$, with $\gamma_{0}$, $\gamma_{1}$, $\gamma_{2}$ and $\gamma_{3}$ its boundary curves. See figure \ref{RegularNeighbourhoodForSphericalFareyfig1}. Note that, since $g_{1} \geq 3$, at least three of the $\gamma_{i}$ must be different.
 \begin{figure}[h]
 \begin{center}
  \resizebox{5cm}{!}{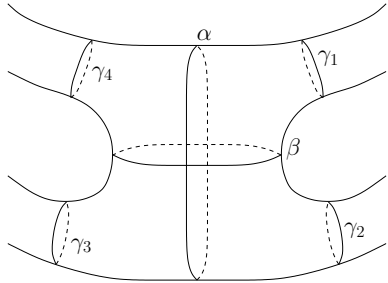} \caption{The regular neighbourhood $N$, the curves $\alpha$ and $\beta$, and the boundary curves $\gamma_{0}$, $\gamma_{1}$, $\gamma_{2}$, $\gamma_{3}$.} \label{RegularNeighbourhoodForSphericalFareyfig1}
 \end{center}
\end{figure}\\
\indent We separate this part of the proof according to whether $\gamma_{i}$ is connected outside $N$ to $\gamma_{i+1}$.\\
\textit{Subcase 1:} If $\gamma_{i}$ is connected outside $N$ to $\gamma_{i+1}$, we can use a proper arc in $S \backslash N$ with endpoints in $\gamma_{i}$ and $\gamma_{i+1}$, to find a curve $\delta$ such that $\beta$ and $\delta$ are disjoint and $i(\alpha,\delta) = 1$. See figure \ref{RegularNeighbourhoodForSphericalFareyfig2}.
\begin{figure}[h]
 \begin{center}
  \resizebox{5cm}{!}{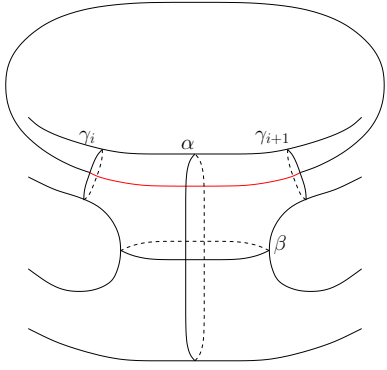} \caption{The regular neighbourhood $N$, the curves $\alpha$ and $\beta$, the arc that connects $\gamma_{i}$ to $\gamma_{i+1}$ outside $N$ (in black), and the arc in $N$ that completes the curve $\delta$.} \label{RegularNeighbourhoodForSphericalFareyfig2}
 \end{center}
\end{figure}\\
\indent Then $\varphi(\delta)$ is disjoint from $\varphi(\beta)$ and $i(\varphi(\alpha),\varphi(\delta)) \neq 0$ (since $\varphi$ is a toroidal-Farey map); thus $\varphi(\alpha) \neq \varphi(\beta)$ and so $i(\varphi(\alpha),\varphi(\beta)) \neq 0$.\\
\textit{Subcase 2:} Let $\gamma_{i}$ not be connected outside $N$ to $\gamma_{i+1}$. Since $g_{1} \geq 3$, $S_{1} \backslash N$ must have a connected component of genus at least $1$. Then, let $\delta$ be a curve disjoint from $\beta$ that is a spherical Farey neighbour of $\alpha$ and satisfies the conditions of the previous subcase. See Figure \ref{toroidalFarey}. Thus $\varphi(\beta)$ is disjoint from $\varphi(\delta)$ and $i(\varphi(\alpha),\varphi(\delta)) \neq 0$; hence $\varphi(\alpha) \neq \varphi(\beta)$ and so $i(\varphi(\alpha),\varphi(\beta)) \neq 0$.
\begin{figure}[h]
\begin{center}
 \includegraphics[width=4.5cm]{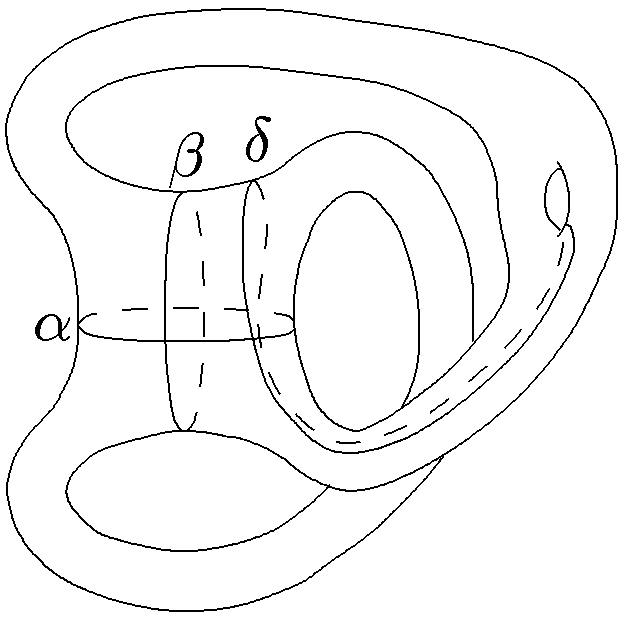}\hspace{1cm}
 \includegraphics[width=6cm]{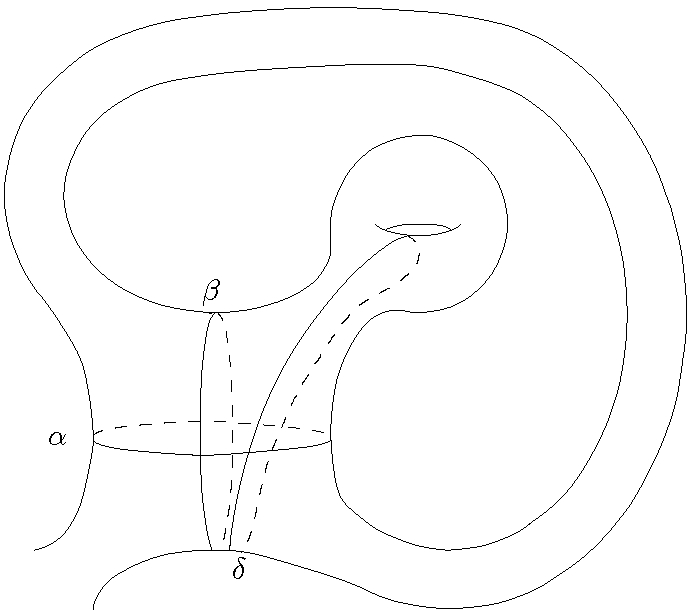}\\
 \includegraphics[width=6cm]{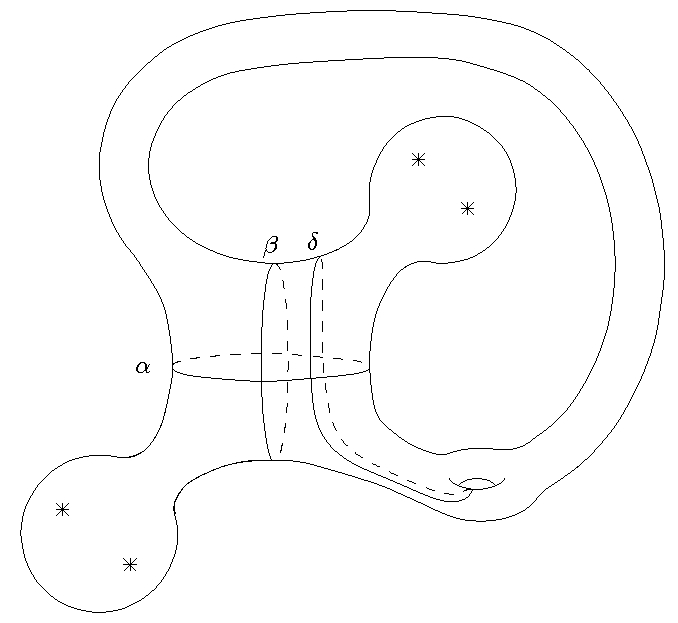}\hspace{1cm}
 \includegraphics[width=6cm]{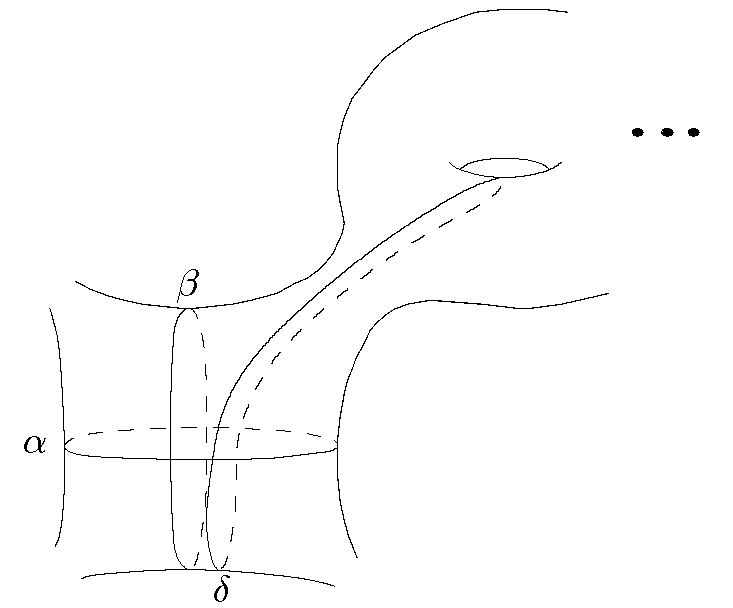}\caption{Examples of curves $\delta$ for the second subcase.} \label{toroidalFarey}
\end{center}
\end{figure}\\
\indent Therefore, $\varphi$ is both a toroidal-Farey map and a spherical-Farey map, as desired.
\end{proof}
%%%%%%%%%%%%
\subsection{Topological data from $\varphi$}\label{chap3sec3subsec2}
\indent Throughout this section we assume $S_{1} = S_{g_{1},n_{1}}$ and $S_{2} = S_{g_{2},n_{2}}$ are orientable surfaces of finite topological type, with $g_{1} \geq 3$, empty boundary, and $n_{1},n_{2} \geq 0$ punctures, such that $\kappa(S_{1}) \geq \kappa(S_{2})$. We also suppose $\fun{\varphi}{\ccomp{S_{1}}}{\ccomp{S_{2}}}$ to be an edge-preserving map.\\
\indent Armed with Lemma \ref{toroidalSpherical}, in this subsection we try to obtain enough topological data from $\varphi$ to force $S_{2}$ to be homeomorphic to $S_{1}$. As was mentioned at the beginning of this chapter, note that many of the lemmas in this subsection are quite similar to those in \cite{Shack}, and while several of the proofs are analogous, others are quite different.\\
\indent Let $P$ be a pants decomposition in $S_{i}$ for some $i \in \{1,2\}$. A \textit{pair of pants subsurface induced by} $P$ is a subsurface of $S_{i}$ whose interior is homeomorphic to $S_{0,3}$ and all its bounding curves are elements of $P$. See figure \ref{ExamplesPairofpantsInducedbyPfig1}.
\begin{figure}[h]
 \begin{center}
  \resizebox{10cm}{!}{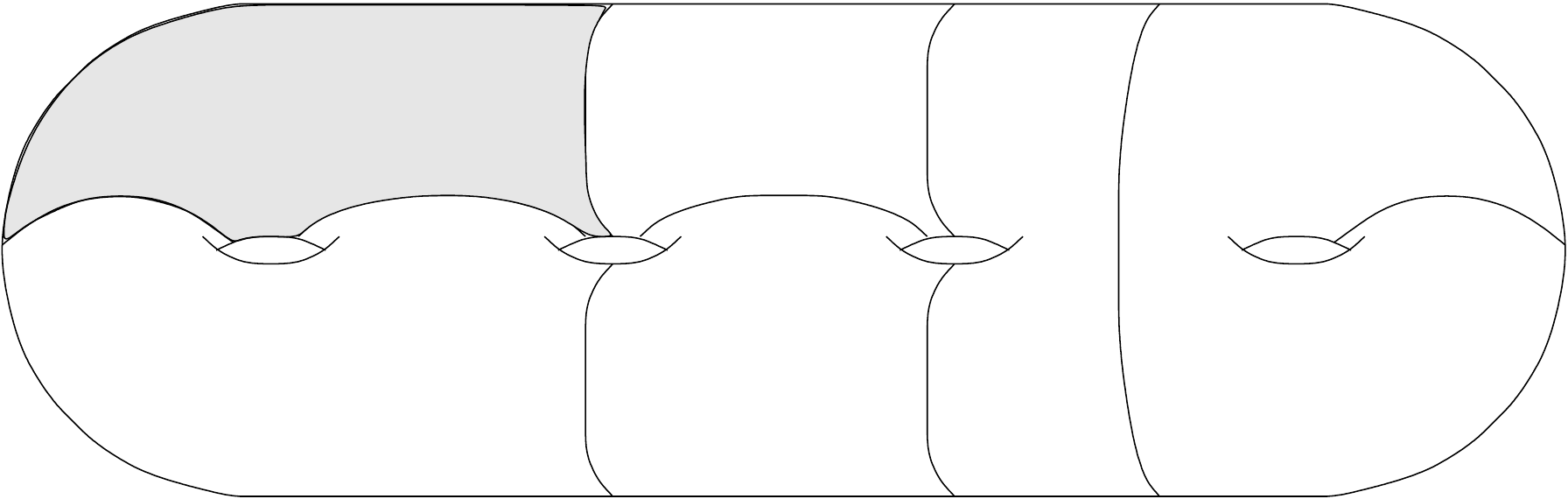}\\[0.3cm]
  \resizebox{10cm}{!}{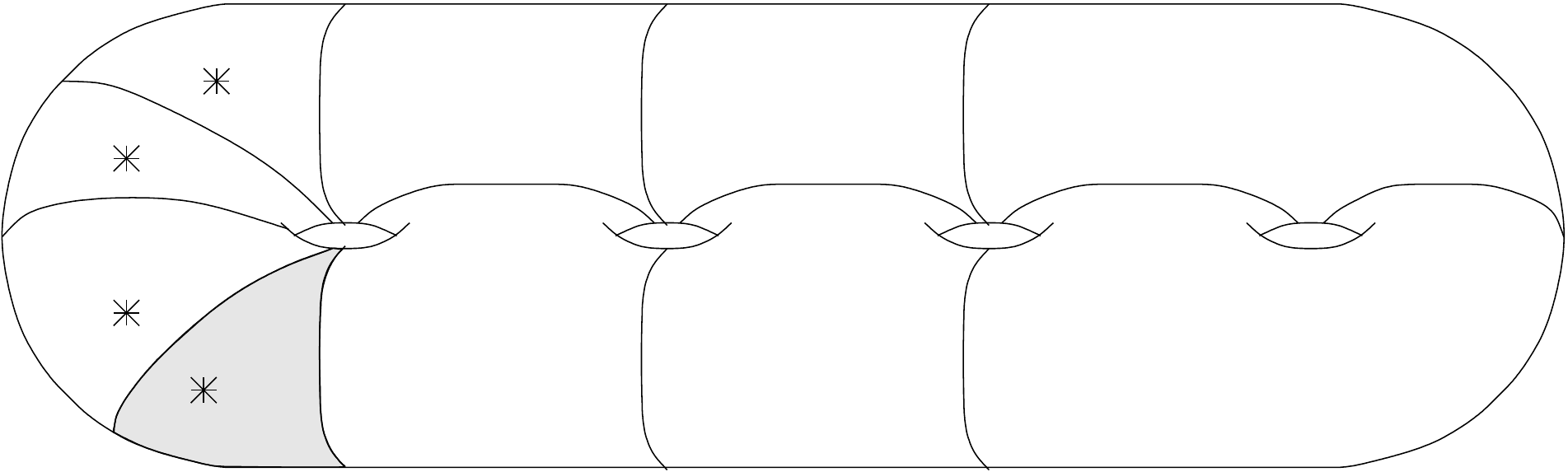} \caption{Two examples of pairs of pants (shaded) induced by pants decompositions in closed surfaces (above) and punctured surfaces (below).} \label{ExamplesPairofpantsInducedbyPfig1}
 \end{center}
\end{figure}\\
\indent Let $\alpha, \beta \in P$. We say $\alpha$ and $\beta$ are adjacent with respect to $P$ if there exists a pair of pants subsurface induced by $P$, that has $\alpha$ and $\beta$ as two of its boundary curves. We define the \textit{adjacency graph of} $P$, denoted by $\acomp{P}$, as the simplicial graph whose vertex set is $P$, and two vertices span an edge if they are adjacent \wrt $P$. The adjacency graph was first introduced by Behrstock and Margalit in \cite{BehrMar}. Afterwards it was used by Shackleton in \cite{Shack}, and we use it in a similarly.\\
\indent Any edge-preserving map $\fun{\varphi}{\ccomp{S_{1}}}{\ccomp{S_{2}}}$ induces a bijective map $\varphi_{P}$ from the vertex set of $\acomp{P}$ to the vertex set of $\acomp{\varphi(P)}$, defined by $\alpha \mapsto \varphi(\alpha)$. We prove that this map is a simplicial isomorphism.
\begin{Lema}\label{adjacencyCCOMP}
 Let $P$ be a pants decomposition of $S_{1}$, and $\alpha, \beta \in P$. We have that $\alpha$ and $\beta$ are adjacent \wrt $P$ if and only if $\varphi(\alpha)$ and $\varphi(\beta)$ are adjacent \wrt $\varphi(P)$. In particular $\varphi_{P}$ is a simplicial isomorphism.
\end{Lema}
\begin{proof}
 Suppose that $\alpha$ and $\beta$ are adjacent \wrt $P$. Then we can find a curve $\gamma \in \ccomp{S_{1}}$ that is a Farey neighbour with both $\alpha$ and $\beta$, and $\gamma$ is disjoint from every element in $P \backslash \{\alpha,\beta\}$. By Corollary \ref{EdgePresToroidal} $\varphi(\gamma)$ is disjoint from every element in $\varphi(P \backslash \{\alpha, \beta\})$ and intersects both $\varphi(\alpha)$ and $\varphi(\beta)$. Then, $\varphi(\gamma)$, $\varphi(\alpha)$ and $\varphi(\beta)$ are curves in $S_{2} \backslash \varphi(P \backslash \{\alpha,\beta\})$.\\
 \indent But if $\varphi(\alpha)$ and $\varphi(\beta)$ are not adjacent \wrt $\varphi(P)$ then they are in different connected components of $S_{2} \backslash \varphi(P \backslash \{\alpha,\beta\})$ while at the same time being intersected by a curve in $S_{2} \backslash \varphi(P \backslash \{\alpha, \beta\})$, which is impossible. Thus, $\varphi(\alpha)$ and $\varphi(\beta)$ are adjacent \wrt $\varphi(P)$.\\
 \indent Conversely, if $\alpha$ and $\beta$ are not adjacent \wrt $P$, let $\gamma_{1}, \gamma_{2} \in \ccomp{S_{1}}$ be such that:
 \begin{enumerate}
  \item $\gamma_{1}$ and $\alpha$ are Farey neighbours.
  \item $\gamma_{2}$ and $\beta$ are Farey neighbours.
  \item $\gamma_{1}$ is disjoint from every element in $(P \backslash \{\alpha\}) \cup \{\gamma_{2}\}$.
  \item $\gamma_{2}$ is disjoint from every element in $(P \backslash \{\beta\}) \cup \{\gamma_{1}\}$.
 \end{enumerate}
 \indent This implies, by Corollary \ref{EdgePresToroidal}, that:
 \begin{enumerate}
  \item $i(\varphi(\alpha), \varphi(\gamma_{1})) \neq 0 \neq i(\varphi(\beta), \varphi(\gamma_{2}))$.
  \item $\varphi(\gamma_{1})$ is disjoint from every element in $\varphi(P \backslash \{\alpha\} \cup \{\gamma_{2}\})$.
  \item $\varphi(\gamma_{2})$ is disjoint from every element in $\varphi(P \backslash \{\beta\} \cup \{\gamma_{1}\})$.
 \end{enumerate}
 \indent By construction $\varphi(P \backslash \{\alpha\})$ has $\kappa(S_{1}) -1 = \kappa(S_{2})-1$ elements, thus $S_{2} \backslash \varphi(P \backslash \{\alpha\})$ is the disjoint union of surfaces homeomorphic to $S_{0,3}$ and exactly one surface of positive complexity; given that $\varphi(\alpha)$ and $\varphi(\gamma_{1})$ are disjoint from every element in $\varphi(P \backslash \{\alpha\})$, then $\varphi(\alpha)$ and $\varphi(\gamma_{1})$ are contained in a complexity-one subsurface of $S_{2}$. Analogously for $\varphi(\beta)$ and $\varphi(\gamma_{2})$. But if $\varphi(\alpha)$ and $\varphi(\beta)$ are adjacent \wrt $\varphi(P)$, we would get that $\varphi(\gamma_{1})$ would intersect $\varphi(\gamma_{2})$, which is not possible.\\
 \indent Therefore $\alpha$ and $\beta$ are adjacent \wrt $P$ if and only if $\varphi(\alpha)$ and $\varphi(\beta)$ are adjacent \wrt $\varphi(P)$. This particularly implies that $\varphi_{P}$ is a bijective simplicial map with simplicial inverse, and so it is an isomorphism.
\end{proof}
\indent An \textit{outer curve} $\gamma \in \ccomp{S_{i}}$ (for some $i \in \{1,2\}$) is a separating curve such that $S_{i} \backslash \{\gamma\}$ has a connected component homeomorphic to $S_{0,3}$.\\
\indent Note that given a pants decomposition $P$ and $\alpha \in P$, we have that $\alpha$ is a nonouter separating curve if and only if the vertex corresponding to $\alpha$ in $\acomp{P}$ is a cut point. As an immediate consequence of this and Lemma \ref{adjacencyCCOMP} we have the following lemma.
\begin{Lema}[cf. Lemma 5 in \cite{Shack}]\label{NonouterNonouter}
 A curve $\alpha$ in $S_{1}$ is a nonouter separating curve if and only if $\varphi(\alpha)$ is a nonouter separating curve in $S_{2}$.
\end{Lema}
\indent With this we have proved that $\varphi$ respects the topological type of nonouter separating curves. To prove a similar result for nonseparating curves we note first the following.
\begin{Rem}\label{OuterDegree}
 Given any outer curve $\alpha$ in $S_{i}$ (for $i=1,2$) and any pants decomposition $P$ of $S_{i}$ such that $\alpha \in P$, we have that $\alpha$ has degree at most $2$ in $\acomp{P}$
\end{Rem}
\begin{Lema}[cf. Lemma 5 in \cite{Shack}]\label{NonsepNonsep}
 If $\alpha$ is a nonseparating curve in $S_{1}$, then $\varphi(\alpha)$ is a nonseparating curve in $S_{2}$.
\end{Lema}
\begin{proof}
 Given a nonseparating curve $\alpha$ in $S_{1}$ we can find a pants decomposition $P$ in $S_{1}$ such that $\alpha$ has degree $4$ in $\acomp{P}$ (see Figure \ref{ExampleAlphaDegree4Sec3fig1}).
 \begin{figure}[h]
  \begin{center}
   \resizebox{10cm}{!}{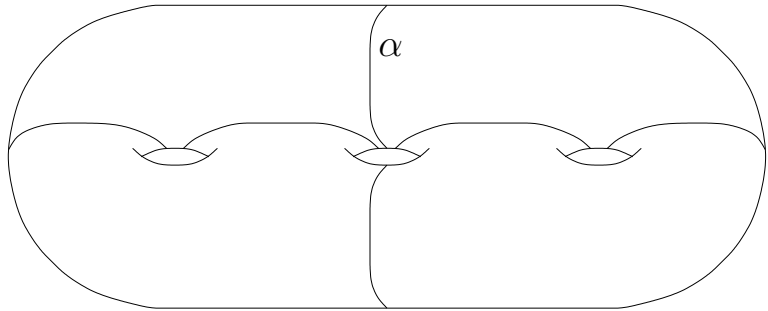} \caption{An example of a pants decomposition in which $\alpha$ has degree $4$.} \label{ExampleAlphaDegree4Sec3fig1}
  \end{center}
 \end{figure}\\
 \indent Due to Lemma \ref{NonouterNonouter}, $\varphi(\alpha)$ cannot be a nonouter separating curve, and by Lemma \ref{adjacencyCCOMP} we have that $\varphi(\alpha)$ has degree $4$ in $\acomp{\varphi(P)}$; so $\varphi(\alpha)$ cannot be an outer curve and therefore $\varphi(\alpha)$ is a nonseparating curve.
\end{proof}
\indent For an analogous result for outer curves, we must give first a brief definition and a remark.\\
\indent A \textit{peripheral pair} $\{\alpha,\beta\}$ in $S_{i}$ (for $i=1,2$) is a multicurve such that $\alpha \neq \beta$, $\alpha$ and $\beta$ are \textbf{nonseparating curves}, and $S_{i}$ has a subsurface with $\alpha$ and $\beta$ as its only boundary curves and whose interior is homeomorphic to $S_{0,3}$ ($\alpha$ and $\beta$ cobord a punctured annulus). See Figure \ref{ExampleDefPeripheralPairs} for an example.
\begin{figure}[h]
 \begin{center}
  \resizebox{10cm}{!}{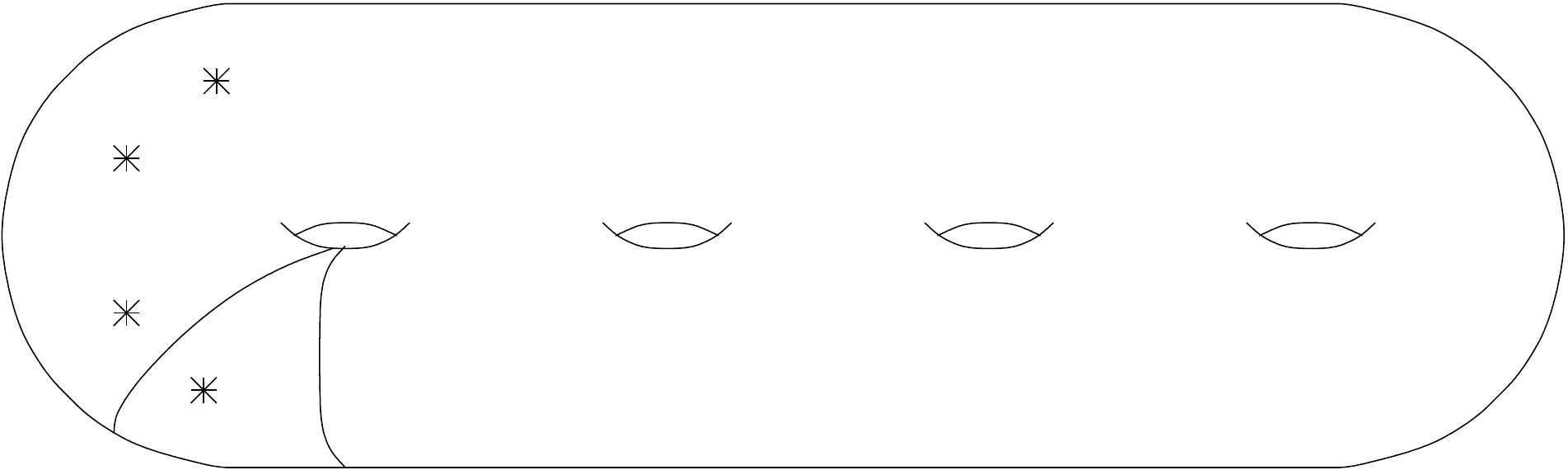} \caption{An example of a peripheral pair in $S_{4,4}$.} \label{ExampleDefPeripheralPairs}
 \end{center}
\end{figure}\\
\indent This leads to the following remark.
\begin{Rem}\label{remarkdeg2}
 Let $P$ be a pants decomposition in $S_{i}$ (for $i=1,2$). If $\alpha$ is a nonseparating curve with degree $2$ in $\acomp{P}$ such that both vertices adjacent to $\alpha$ \wrt $P$ (say $\beta$ and $\gamma$) are nonseparating curves, then $\{\alpha,\beta\}$ and $\{\alpha,\gamma\}$ are peripheral pairs in $S_{i}$.
\end{Rem}
\indent Armed with this remark we can prove that $\varphi$ respects the topological type of outer curves (and thus of all curves in $S_{1}$). In particular, the proof of the following lemma is quite different from the proof of the analogous statement in \cite{Shack}, due to the different approaches.
\begin{Lema}\label{OuterOuter}
 If $\alpha$ is an outer curve in $S_{1}$, then $\varphi(\alpha)$ is an outer curve in $S_{2}$.
\end{Lema}
\begin{proof}
 Let $\alpha$ be an outer curve in $S_{1}$. Let $P$ a pants decomposition in $S_{1}$ with $\alpha \in P$ such that the degree of $\alpha$ in $\acomp{P}$ is $2$ and every element in $P \backslash \{\alpha\}$ is a nonseparating curve. See Figure \ref{ExampleOuterCurveDegree2WithNonsepNeighbours} for an example.
\begin{figure}[h]
 \begin{center}
  \resizebox{10cm}{!}{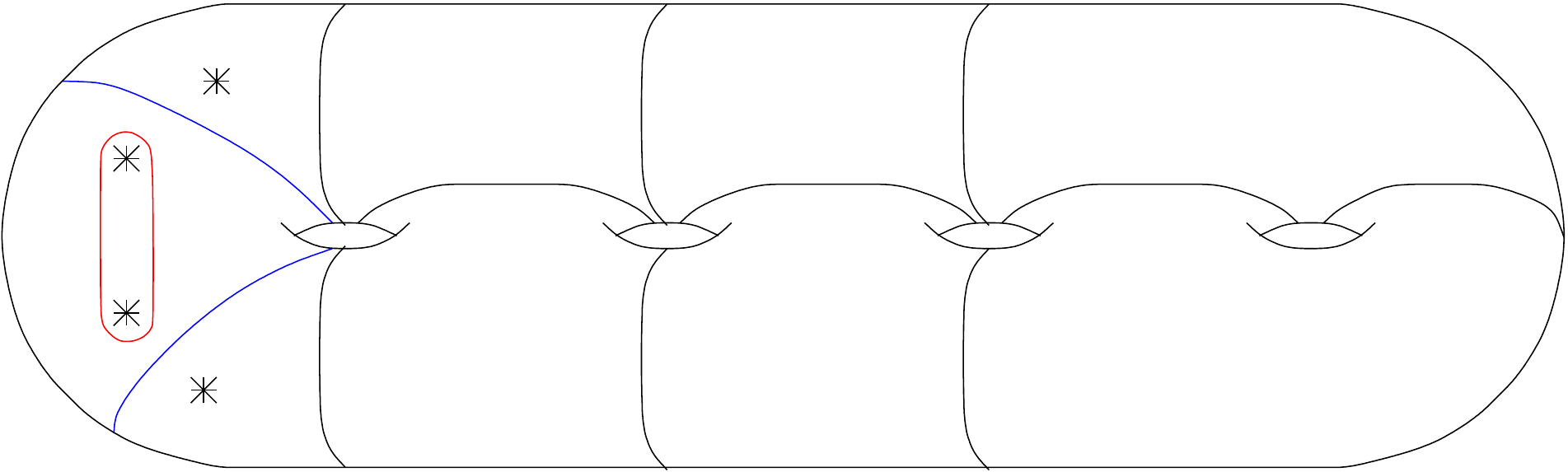} \caption{An example of curve $\alpha$ in red, the pants decomposition $P$ in black, and the curves $\beta$ and $\gamma$ adjacent to $\alpha$ \wrt $P$ in blue.} \label{ExampleOuterCurveDegree2WithNonsepNeighbours}
 \end{center}
\end{figure}\\
\indent Then by Lemma \ref{adjacencyCCOMP}, $\varphi(\alpha)$ has degree $2$ in $\acomp{\varphi(P)}$. Due to Lemma \ref{NonouterNonouter}, if $\varphi(\alpha)$ were not an outer curve it would be a nonseparating curve; also, by Lemma \ref{NonsepNonsep}, every element in $\varphi(P)$ would then be a nonseparating curve. Let $\beta$ and $\gamma$ be the two nonseparating curves adjacent to $\alpha$ \wrt $P$, which then are also adjacent to each other \wrt $P$. It follows by Remark \ref{remarkdeg2}, that $\{\varphi(\beta), \varphi(\alpha)\}$ and $\{\varphi(\gamma), \varphi(\alpha)\}$ are peripheral pairs.\\
 \indent Given that $\varphi(\beta)$ and $\varphi(\gamma)$ are also adjacent \wrt $\varphi(P)$, there exists a subsurface $P^{\prime}$ in $S_{2}$ whose interior is homeomorphic to $S_{0,3}$ and has $\varphi(\beta)$ and $\varphi(\gamma)$ as two of its boundary curves. Let $\delta$ be a (possibly nonessential) curve in $S_{2}$ contained in $P^{\prime}$ that is isotopic neither to $\varphi(\beta)$ nor to $\varphi(\gamma)$. If $\delta$ is nonessential, we would have that $\kappa(S^{\prime}) < 4$ (see Figure \ref{OutertoOuterCurveDeltafig1}), but if it is essential it would have to be a separating curve in $\varphi(P)$ which is impossible, reaching like this a contradiction. Therefore $\varphi(\alpha)$ is an outer curve.
\end{proof}
\begin{figure}[h]
 \begin{center}
  \resizebox{8 cm}{!}{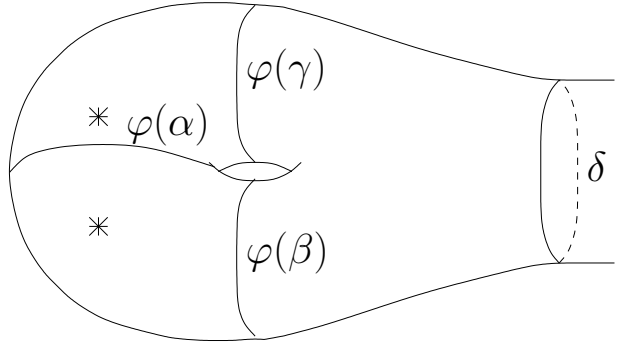} \caption{The peripheral pairs $\{\varphi(\beta), \varphi(\alpha)\}$ and $\{\varphi(\alpha), \varphi(\gamma)\}$, with the curve $\delta$.} \label{OutertoOuterCurveDeltafig1}
 \end{center}
\end{figure}
\indent This lemma gives us the following information concerning the punctures of $S_{1}$ and $S_{2}$.
\begin{Cor}\label{PuncturesOuter}
 If $n_{1}$ is even, then $n_{1} \leq n_{2}$; if $n_{1}$ is odd, then $n_{1} -1 \leq n_{2}$.
\end{Cor}
\begin{proof}
 If $n_{1} = 0$ or $n_{1} = 1$, we obtain the desired result from $n_{2}$ being nonnegative.\\
 \indent If $n_{1} \geq 2$ and it is even, there exists $k \in \mathbb{Z}^{+}$ such that $2k = n_{1}$. Let $\{\alpha_{1}, \ldots, \alpha_{k}\}$ be a multicurve comprised of only outer curves. By Lemma \ref{OuterOuter}, we have that $\{\varphi(\alpha_{1}), \ldots, \varphi(\alpha_{k})\}$ is a multicurve of cardinality $k$ comprised of only outer curves. As such, $S_{2}$ must have at least $2k = n_{1}$ punctures.\\
 \indent If $n_{1} \geq 2$ and it is odd, there exists $k \in \mathbb{Z}^{+}$ such that $2k +1 = n_{1}$. From there we proceed analogously to the previous case to deduce that $S_{2}$ can contain $k$ outer curves, thus having at least $2k = n_{1} -1$ punctures.
\end{proof}
\indent With this corollary we need only a similar comparison between $g_{1}$ and $g_{2}$ to try and deduce that $S_{1}$ is homeomorphic to $S_{2}$. For this, we must first prove some technical lemmas, including the preservation of intersection number $1$ under $\varphi$.\\
\indent Let $\alpha$, $\beta$ and $\gamma$ be three distinct curves in $S_{i}$ (for $i=1,2$). We say $\alpha$, $\beta$ and $\gamma$ bound a pair of pants in $S_{i}$ if there is a subsurface of $S_{i}$ whose interior is homeomorphic to $S_{0,3}$ and has $\{\alpha,\beta,\gamma\}$ as its three boundary curves. We proceed to prove this is preserved under $\varphi$.
\begin{Lema}\label{BoundBound}
 If $\alpha$, $\beta$ and $\gamma$ are three distinct nonseparating curves in $S_{1}$ that bound a pair of pants in $S_{1}$, then $\varphi(\alpha)$, $\varphi(\beta)$ and $\varphi(\gamma)$ bound a pair of pants in $S_{2}$.
\end{Lema}
\begin{proof}
 Given that $\alpha$, $\beta$ and $\gamma$ are nonseparating curves, let $P$ be a pants decomposition comprised of only nonseparating curves such that $\alpha, \beta, \gamma \in P$, $\alpha$ and $\beta$ have degree three in $\acomp{P}$, $\gamma$ has degree four in $\acomp{P}$, and $\beta$ is the only curve in $P$ that is adjacent \wrt $P$ to both $\alpha$ and $\gamma$. See figure \ref{ExamplePairofpantsAlphaBetaDeg3GammaDeg4fig1} for an example.
\begin{figure}[h]
 \begin{center}
  \resizebox{10cm}{!}{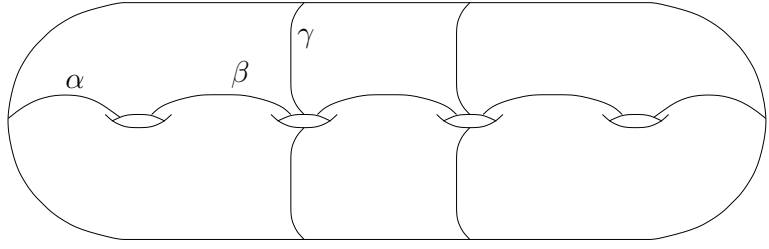} \caption{An example of a pair of pants $P$ with $\alpha$ and $\beta$ having degree $3$ in $\acomp{P}$ and $\gamma$ degree $4$ in $\acomp{P}$.} \label{ExamplePairofpantsAlphaBetaDeg3GammaDeg4fig1}
 \end{center}
\end{figure}\\
 \indent By Lemma \ref{NonsepNonsep} we have that $\varphi(P)$ is comprised of only nonseparating curves, and by Lemma \ref{adjacencyCCOMP} we have that $\varphi(\alpha)$ and $\varphi(\beta)$ have degree three and $\varphi(\gamma)$ has degree four in $\acomp{\varphi(P)}$. If $\varphi(\alpha)$, $\varphi(\beta)$ and $\varphi(\gamma)$ do not bound a pair of pants on $S_{2}$ then there exist a pair of pants bounded by $\varphi(\alpha)$, $\varphi(\beta)$ and $\delta_{1} \neq \varphi(\gamma)$, another pair of pants bounded by $\varphi(\beta)$, $\varphi(\gamma)$ and $\delta_{2} \neq \varphi(\alpha)$, and another pair of pants bounded by $\varphi(\alpha)$, $\varphi(\gamma)$ and $\delta_{3} \neq \varphi(\beta)$. See figure \ref{ExamplePairofpantsAlphaBetaDeg3GammaDeg4fig2} for an example.
\begin{figure}[h]
 \begin{center}
  \resizebox{7cm}{!}{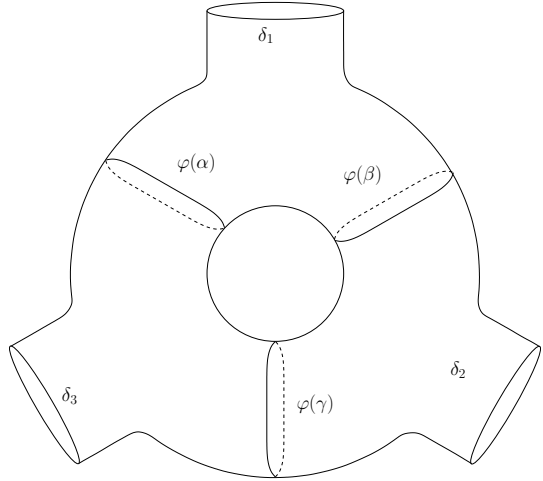}\caption{The case in which $\varphi(\alpha)$, $\varphi(\beta)$ and $\varphi(\gamma)$ do not bound a pair of pants in $S_{2}$.} \label{ExamplePairofpantsAlphaBetaDeg3GammaDeg4fig2}
 \end{center}
\end{figure}\\
 \indent Note that $\delta_{1}$, $\delta_{2}$ and $\delta_{3}$ are neither necessarily distinct nor necessarily essential (they could be boundary curves of a neighbourhood of some puncture).\\
 \indent Once again, by Lemma \ref{adjacencyCCOMP}, $\varphi(\beta)$ is the only curve in $\acomp{\varphi(P)}$ that is adjacent \wrt $\varphi(P)$ to both $\varphi(\alpha)$ and $\varphi(\gamma)$; this implies that $\delta_{3}$ is not an essential curve, but this leads us to a contradiction, since $\varphi(\gamma)$ would then have degree at most $3$.
\end{proof}
\indent To prove that $\varphi$ preserves intersection number $1$, we must first recall Irmak's characterization of intersection number $1$ in \cite{Irmak3}. We have modified the statement to suit the notation used here.
\begin{Lema}[2.7 in \cite{Irmak3}] \label{CharactIrmak}
 Let $S$ be a surface homeomorphic to $S_{g,n}$, with $g \geq 2$ and $n \geq 0$; let also $\alpha_{1}$ and $\alpha_{2}$ be two nonseparating curves. Then, $i(\alpha_{1}, \alpha_{2}) = 1$ if and only if there exist distinct and nonseparating curves $\alpha_{3}$, $\alpha_{4}$, $\alpha_{5}$, $\alpha_{6}$ and $\alpha_{7}$ such that:
 \begin{enumerate}
  \item $i(\alpha_{i}, \alpha_{j}) = 0$ if and only if the curves $\gamma_{i}$ and $\gamma_{j}$ of Figure \ref{FigLemmaIrmakIntUno} are disjoint.
  \item The curves $\alpha_{1}$, $\alpha_{3}$, $\alpha_{5}$ and $\alpha_{6}$ are such that: $S \backslash \{\alpha_{5}, \alpha_{6}\}$ is disconnected with a connected component homeomorphic to $S_{1,2}$ that contains $\alpha_{1}$ and $\alpha_{2}$, and finally $\{\alpha_{1}, \alpha_{3}, \alpha_{5}\}$ and $\{\alpha_{1}, \alpha_{3}, \alpha_{6}\}$ bound pairs of pants in $S$.
 \end{enumerate}
\end{Lema}
\begin{figure}[h]
 \begin{center}
  \resizebox{8cm}{!}{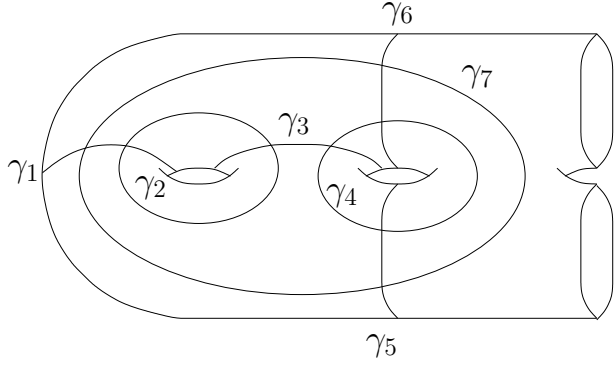} \caption{The curves needed for Lemma \ref{CharactIrmak} to characterize intersection one.} \label{FigLemmaIrmakIntUno}
 \end{center}
\end{figure}
\begin{Lema}\label{IntUno}
 If $\alpha_{1}$ and $\alpha_{2}$ are curves in $S_{1}$ such that $i(\alpha_{1}, \alpha_{2}) = 1$, then $i(\varphi(\alpha_{1}), \varphi(\alpha_{2})) = 1$.
\end{Lema}
\begin{proof}
 Let $\alpha_{1}$ and $\alpha_{2}$ be curves in $S_{1}$ that intersect once. By Lemma \ref{CharactIrmak} there exist curves $\alpha_{3}$, $\alpha_{4}$, $\alpha_{5}$, $\alpha_{6}$ and $\alpha_{7}$ that satisfy (1) and (2). Moreover, we can ask that whenever $\alpha_{i}$ intersects $\alpha_{j}$, (for some $i,j \in \{1, \ldots, 7\}$) they intersect once.\\
 \indent We now prove that $\varphi(\{\alpha_{1}, \ldots \alpha_{7}\})$ also satisfy the conditions of Lemma \ref{CharactIrmak}.\\
 \indent Given that $\varphi$ is an edge-preserving and Farey map, this implies that (1) is preserved under $\varphi$. By Lemma \ref{BoundBound} we have that $\{\varphi(\alpha_{1}), \varphi(\alpha_{3}), \varphi(\alpha_{5})\}$ and $\{\varphi(\alpha_{1}), \varphi(\alpha_{3}), \varphi(\alpha_{6})\}$ bound two distinct pair of pants in $S_{2}$, which implies that $S_{2} \backslash \{\varphi(\alpha_{5}), \varphi(\alpha_{6})\}$ is disconnected and has one connected component, namely $T$, homeomorphic to $S_{1,2}$. Since $i(\varphi(\alpha_{1}), \varphi(\alpha_{2})) \neq 0$ and $\varphi(\alpha_{2})$ is disjoint from both $\varphi(\alpha_{5})$ and $\varphi(\alpha_{6})$, we have that $\varphi(\alpha_{2})$ (and by construction $\varphi(\alpha_{1})$) is contained in $T$.\\
 \indent Therefore, by Lemma \ref{CharactIrmak}, $i(\varphi(\alpha_{1}), \varphi(\alpha_{2})) = 1$.
\end{proof}
\indent The following corollary is a consequence of this lemma.
\begin{Cor}\label{GenusChains}
 Chains in $S_{1}$ are mapped to chains of the same length in $S_{2}$. In particular $g_{1} \leq g_{2}$.
\end{Cor}
\begin{proof}
 Lemma \ref{IntUno} and $\varphi$ being an edge-preserving map imply that chains in $S_{1}$ are mapped to chains of the same length in $S_{2}$. Now, let $C$ be a chain in $S_{1}$ of length $2g_{1}$; then $\varphi(C)$ is a chain of length $2g_{1}$. Hence the regular neighbourhood of $\varphi(C)$ has genus $g_{1}$. Therefore $g_{1} \leq g_{2}$.
\end{proof}
\indent Finally, we prove that $S_{1}$ is homeomorphic to $S_{2}$.
\begin{Lema}\label{S1HomeoS2}
 Let $S_{1} = S_{g_{1},n_{1}}$ and $S_{2} = S_{g_{2},n_{2}}$ be orientable surface of finite topological type, with $g_{1} \geq 3$, empty boundary, and $n_{1},n_{2} \geq 0$ punctures, such that $\kappa(S_{1}) \geq \kappa(S_{2})$; let also $\fun{\varphi}{\ccomp{S_{1}}}{\ccomp{S_{2}}}$ be an edge-preserving map. Then, $S_{1}$ is homeomorphic to $S_{2}$.
\end{Lema}
\begin{proof}
 We divide the proof according to the parity of $n_{1}$.\\
 \indent If $n_{1}$ is even, by Lemma \ref{PuncturesOuter} we have that $n_{1} \leq n_{2}$. Also, by Lemma \ref{GenusChains} we have that $g_{1} \leq g_{2}$. Supposing that $g_{1} < g_{2}$, we obtain the following contradiction: $$\kappa(S_{2}) = \kappa(S_{1}) = 3g_{1} -3 + n_{1} < 3g_{2} -3 + n_{1} \leq 3g_{2} -3 + n_{2} = \kappa(S_{2}).$$
 \indent Thus, $g_{1} = g_{2}$. Given that $\kappa(S_{1}) = \kappa(S_{2})$, this implies that $n_{1} = n_{2}$. Hence $S_{1}$ is homeomorphic to $S_{2}$.\\
 \indent If $n_{1}$ is odd, by Lemma \ref{PuncturesOuter} we have that $n_{1} - 1 \leq n_{2}$. If $n_{1} -1 = n_{2}$, we have the following: $$3g_{1} -3 + n_{1} = \kappa(S_{1}) = \kappa(S_{2}) = 3g_{2} -3 +n_{2} = 3g_{2} -3 +(n_{1} -1),$$ thus, $$3g_{1} = 3g_{2} -1$$ which is impossible since $g_{1}, g_{2} \in \nat$. Hence $n_{1} \leq n_{2}$, and then we proceed as in the previous case.\\
 \indent Therefore, $S_{1}$ is homeomorphic to $S_{2}$.
\end{proof}
%%%%%%%%%%%%
\subsection{Proof of Theorem \ref{TheoB}}\label{chap3sec3subsec3}
\indent In view of Lemma \ref{S1HomeoS2} we can assume then that any result concerning edge-preserving self-maps of $\ccomp{S}$ with $S = S_{g,n}$, $g \geq 3$ and $n \geq 0$, can be applied to $\varphi$, which (a priori) is not a self-map but a map between two curve graphs of surfaces of a priori different topological type.\\
\indent In this subsection we use the rigid set $\X(S_{1})$ from Section \ref{prelim}, we prove that $\varphi|_{\X(S_{1})}$ is injective and thus, by Corollary \ref{Thm4}, $\varphi$ is induced by a homeomorphism.
\begin{proof}[\textbf{Proof of Theorem \ref{TheoB}}]
 \indent To prove that $\varphi|_{\X(S_{1})}$ is injective, it can be verified by inspection that given any two curves $\alpha, \beta \in \X(S_{1})$, we have that exactly one of these statements is satisfied, and these cases are dealt with individually:
\begin{itemize}
 \item $\alpha$ is disjoint from $\beta$.
 \item $i(\alpha, \beta) = 1$.
 \item $i(\alpha, \beta) = 2$ and any regular neighbourhood of $\{\alpha,\beta\}$ is homeomorphic to a four-holed sphere.
 \item $\alpha$ and $\beta$ are separating curves and $i(\alpha,\beta) = 4$.
 \item $\alpha$ and $\beta$ are separating curves and $i(\alpha,\beta) = 8$.
\end{itemize}
\indent Now, let $\alpha$ and $\beta$ be two distinct curves in $\X(S_{1})$. If $\alpha$ and $\beta$ are disjoint or are Farey neighbours, since $\varphi$ is an edge-preserving Farey map, we have that $\varphi(\alpha) \neq \varphi(\beta)$; otherwise, if $\alpha$ and $\beta$ are separating curves intersecting either $4$ or $8$ times, we can always find a curve $\gamma$ such that $\gamma$ and $\beta$ are disjoint, and $\gamma$ and $\alpha$ are Farey neighbours (see Figure \ref{ProofThm5fig1} for examples). Thus, $\varphi(\alpha) \neq \varphi(\beta)$, which implies $\varphi|_{\X(S_{1})}$ is injective, and by Corollary \ref{Thm4} we have that $\varphi$ is induced by a homeomorphism.
\end{proof}
\begin{figure}[h]
 \begin{center}
  \resizebox{6cm}{!}{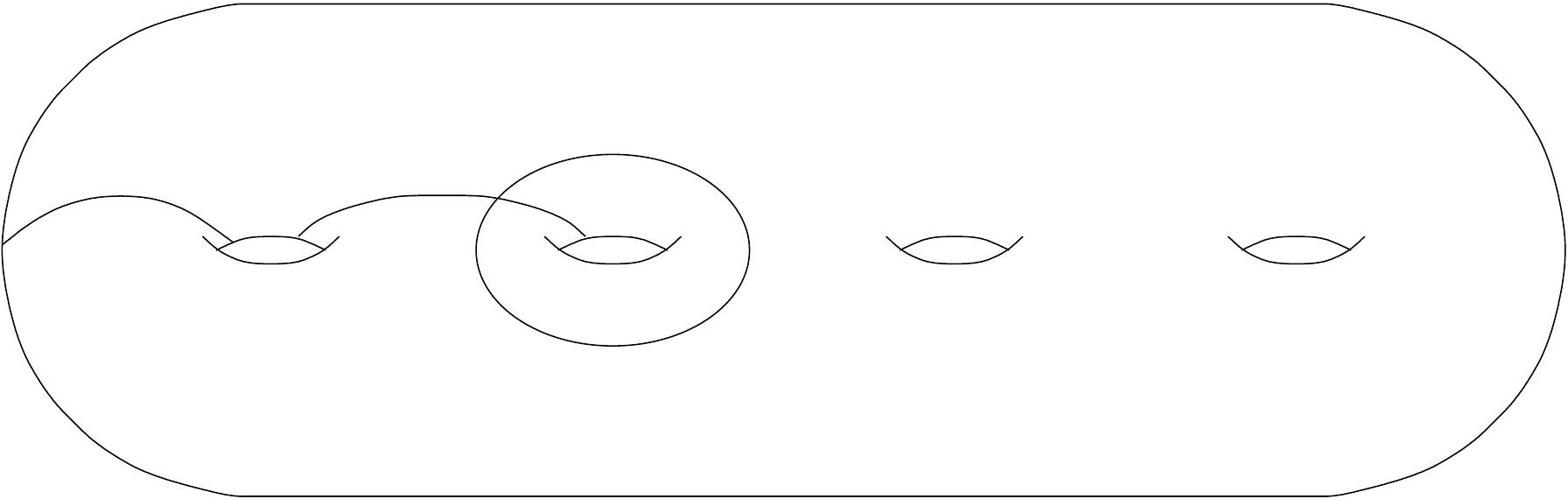}\hspace{0.3cm} \resizebox{6cm}{!}{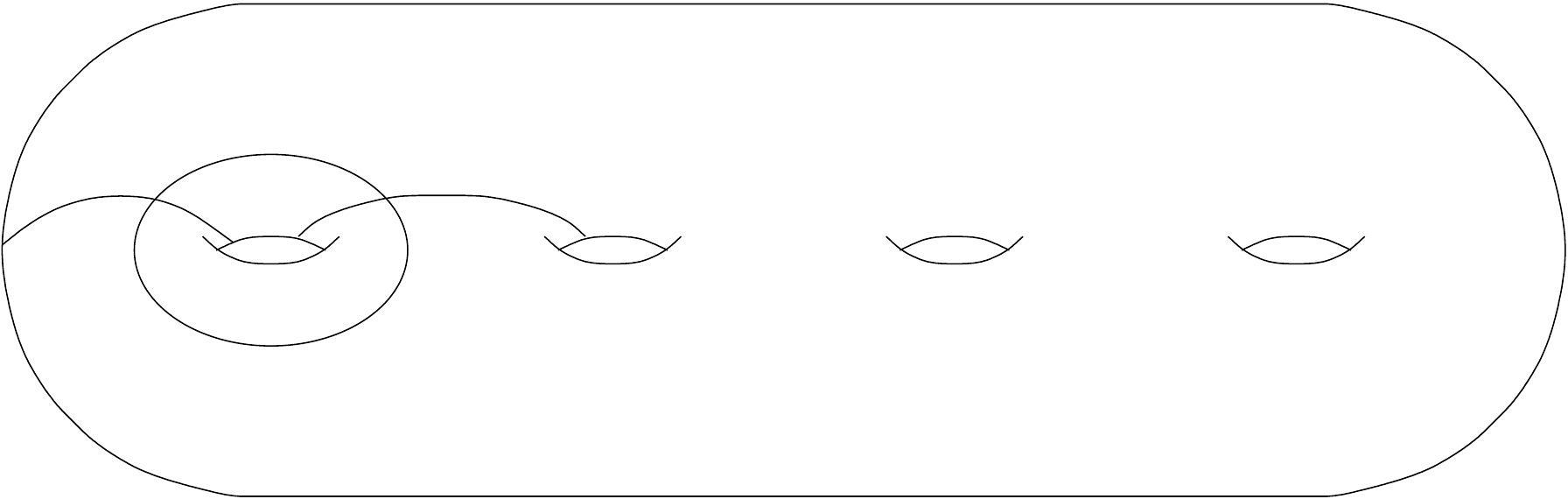}\\[0.3cm]
  \resizebox{6cm}{!}{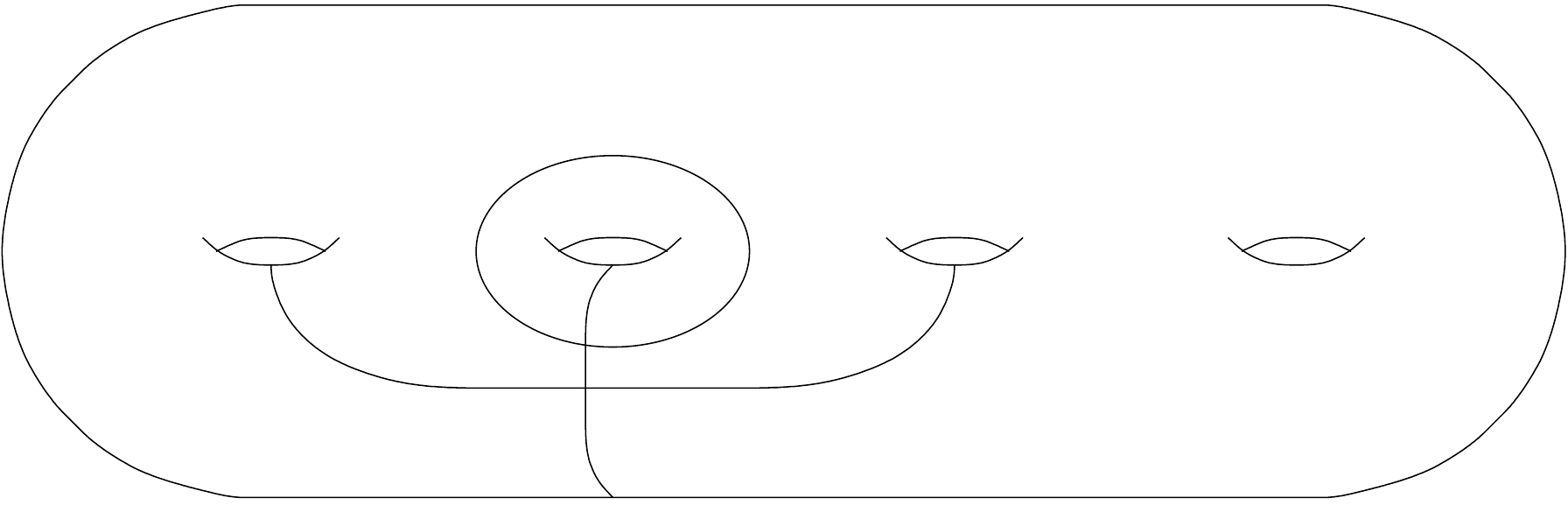}\hspace{0.3cm} \resizebox{6cm}{!}{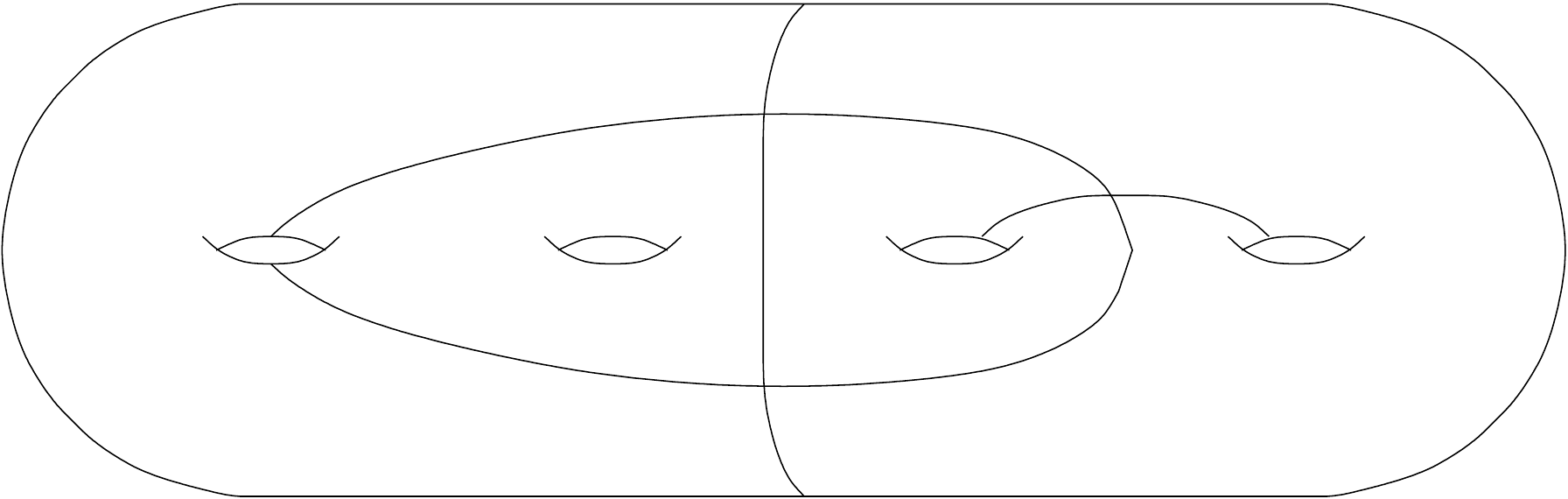}\\[0.3cm]
  \resizebox{6cm}{!}{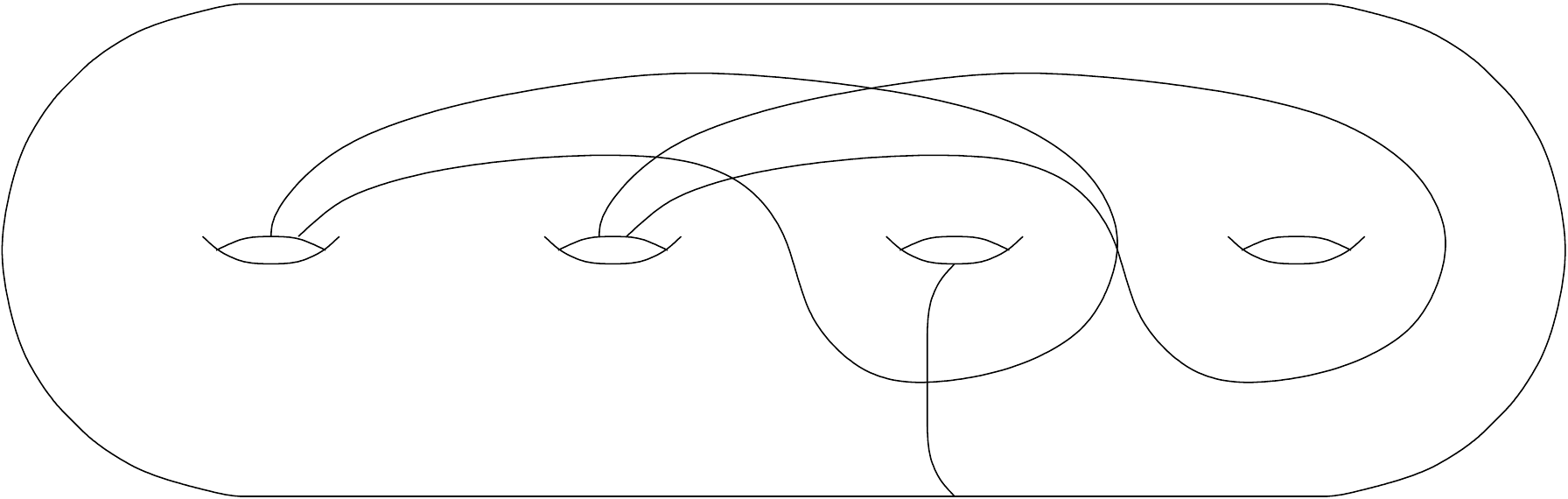} \caption{The curves $\gamma$ needed in the proof of Theorem \ref{TheoB} that are Farey neighbours of $\alpha$ and disjoint from $\beta$.} \label{ProofThm5fig1}
 \end{center}
\end{figure}
%%%%%%%%%%%%%%%%%%%%%%%%
\subsection{Proof of Corollary \ref{IntroCor2}}\label{chap3sec4}
\indent For the sake of convenience, we restate Corollary \ref{IntroCor2}.
\begin{Corno}[\ref{IntroCor2}]\label{Cor2}
 Let $S_{1} = S_{g_{1},n_{1}}$ and $S_{2} = S_{g_{2},n_{2}}$, such that $g_{1} \geq 3$ and $\kappa(S_{1}) \geq \kappa(S_{2}) \geq 6$; let also $\Gamma < \EMod{S_{1}}$ be a subgroup such that for every curve $\gamma$ in $S_{1}$ there exists $N \neq 0$ with $\tau_{\gamma}^{N} \in \Gamma$ (where $\tau_{\gamma}$ denotes the right Dehn twist along $\gamma$), and let $\fun{\phi}{\Gamma}{\EMod{S_{2}}}$ be a homomorphism such that:
 \begin{enumerate}
  \item For each curve $\gamma$ in $S_{1}$, there exist $L, M \in \mathbb{Z}^{*}$ such that $\tau_{\gamma}^{L} \in \Gamma$ and $\phi(\tau_{\gamma}^{L}) = \tau_{\delta}^{M}$ for some curve $\delta$ in $S_{2}$.
  \item For any disjoint curves $\alpha$ and $\beta$, there exist $N_{\alpha}, N_{\beta} \neq 0$ such that the subgroup generated by $\phi(\tau_{\alpha}^{N_{\alpha}})$ and $\phi(\tau_{\beta}^{N_{\beta}})$, is not cyclic.
 \end{enumerate}
 Then, $S_{1}$ is homeomorphic to $S_{2}$ and $\phi$ is the restriction to $\Gamma$ of an inner automorphism of $\EMod{S}$ with $S \cong S_{1} \cong S_{2}$. 
\end{Corno}
\begin{proof}
 We first induce a \textbf{simplicial} map $\fun{\varphi}{\ccomp{S_{1}}}{\ccomp{S_{2}}}$ from $\phi$.\\
 \indent Given a curve $\gamma$ in $S_{1}$, we define $\varphi(\gamma)$ as the curve $\delta$ such that $\phi(\tau_{\gamma}^{N}) = \tau_{\delta}^{M}$. This is well-defined.\\
\indent Let $\alpha$ and $\beta$ be curves in $S_{1}$; recall that for $N,M \in \mathbb{Z} \backslash \{0\}$, $\tau_{\alpha}^{N} \tau_{\beta}^{M} = \tau_{\beta}^{M}\tau_{\alpha}^{N}$ if and only if $i(\alpha,\beta) = 0$ (see \cite{FarbMar}). Then if $i(\alpha,\beta) = 0$, we have that $$\tau_{\varphi(\alpha)}^{M} \tau_{\varphi(\beta)}^{M^{\prime}} = \phi(\tau_{\alpha}^{N}) \phi(\tau_{\beta}^{N^{\prime}}) = \phi(\tau_{\alpha}^{N} \tau_{\beta}^{N^{\prime}}) = \phi(\tau_{\beta}^{N^{\prime}} \tau_{\alpha}^{N}) = \phi(\tau_{\beta}^{N^{\prime}}) \phi(\tau_{\alpha}^{N}) = \tau_{\varphi(\beta)}^{M^{\prime}} \tau_{\varphi(\alpha)}^{M},$$ therefore $i(\varphi(\alpha),\varphi(\beta)) = 0$ and $\varphi$ is simplicial.\\
\indent Now we need to prove that $\varphi$ is an \textbf{edge-preserving map} so we can apply Theorem \ref{TheoB}.\\
\indent To do so, we only need to prove that if $\alpha$ and $\beta$ are disjoint, then $\varphi(\alpha) \neq \varphi(\beta)$. We prove this by contradiction and suppose $\alpha$ and $\beta$ are disjoint curves such that $\varphi(\alpha) = \varphi(\beta)$. Then we have that the group $\left\langle \tau_{\varphi(\alpha)}^{M}, \tau_{\varphi(\beta)}^{M^{\prime}} \right\rangle = \left\langle \tau_{\varphi(\alpha)}^{M}, \tau_{\varphi(\alpha)}^{M^{\prime}}\right\rangle$ is cyclic which contradicts the hypothesis. Therefore $\varphi$ is edge-preserving.\\
\indent By Theorem \ref{TheoB}, $S_{1}$ is homeomorphic to $S_{2}$ and we have that there exists $f \in \EMod{S}$ such that $\varphi(\gamma) = f(\gamma)$ for all curves $\gamma$ in $S$, letting $S \cong S_{1} \cong S_{2}$. This implies that for some $N$, there exists $M$ such that $\phi(\tau_{\gamma}^{N}) = \tau_{f(\gamma)}^{M}$.\\
\indent Recall (see \cite{FarbMar}) that for any curve $\gamma$ in $S$, any $f \in \EMod{S}$ and $N \in \mathbb{Z}$, we have that $$f \tau_{\gamma}^{N}f^{-1} = \tau_{f(\gamma)}^{N^{\prime}},$$ where $N^{\prime} = N$ if $f$ an orientation preserving mapping class and $N^{\prime} = -N$ otherwise.\\
\indent Finally we proceed as in Ivanov's Theorem 2 in \cite{Ivanov}. Let $g \in \Gamma$, $\gamma$ be a curve in $S$, and $N$ be such that $\varphi(\tau_{\gamma}^{N}) = \tau_{f(\gamma)}^{M}$ for some $M \neq 0$. We know that: $$\phi(g \tau_{\gamma}^{N} g^{-1}) = \phi(g) \phi(\tau_{\gamma}^{N}) \phi(g^{-1}) = \phi(g) \tau_{f(\gamma)}^{M} \phi(g)^{-1} = \tau_{\phi(g)(f(\gamma))}^{M^{\prime}}.$$
\indent On the other hand we have: $$\phi(g \tau_{\gamma}^{N} g^{-1}) = \phi(\tau_{g(\gamma)}^{N^{\prime}}),$$ 
so, there exists powers $L, L^{\prime}, K, K^{\prime} \in \mathbb{Z} \backslash \{0\}$ (multiples of $N$, $N^{\prime}$, $M$ and $M^{\prime}$) such that $$\tau_{f(g(\gamma))}^{L^{\prime}} = \phi(\tau_{g(\gamma)}^{L}) = \phi(g \tau_{\gamma}^{K} g^{-1}) = \tau_{\phi(g)(f(\gamma))}^{K^{\prime}}$$
\indent Hence, by the same argument as above,  $L^{\prime} = K^{\prime}$ and more importantly $\phi(g)(f(\gamma)) = f(g(\gamma))$. Since $\gamma = f^{-1}(\alpha)$ for some unique curve $\alpha$ in $S$, we have that $\varphi(g)(\alpha) = f(g(f^{-1}(\alpha)))$ for all curves $\alpha$ in $S$. Since $S$ has genus $g \geq 3$ this implies that $\phi(g) = f \circ g \circ f^{-1}$ as desired.
\end{proof}
\indent This Corollary is indeed an extension of Shackleton's result for surfaces of complexity at least $6$. Any finite index subgroup of $\EMod{S}$ satisfies the conditions on $\Gamma$, but $\Gamma$ could have infinite index (see \cite{HumpPower} and \cite{FunarTQFT} for examples of infinite index subgroups satisfying (1) and (2)); also every homomorphism injective in the stabilizers of every curve in $S_{1}$ satisfies conditions (1) and (2).
%%%%%%%%%%%%%%%%%%%%%%%%%%%%%%%%%%%%%%%%%%%%%%%%%%%%%%%%%%%%
%%%%%%%%%%%%%%%%%%%%%%%%%%%%%%%%%%%%%%%%%%%%%%%%%%%%%%%%%%%%
%%%%%%%%%%%%%%%%%%%%%%%%%%%%%%%%%%%%%%%%%%%%%%%%%%%%%%%%%%%%

\end{document}